\documentclass[11pt]{article}
\usepackage{latexsym,amsmath,amscd,amssymb, graphics}
\usepackage{amsmath,amssymb,mathrsfs,framed,esint,slashed,color}
\usepackage{amsthm}
\usepackage{enumerate}
\usepackage[colorlinks]{hyperref}
\usepackage{enumitem}
\usepackage[margin=1in]{geometry}
\usepackage{stmaryrd}
\usepackage{diagrams}
\usepackage{graphicx}
\usepackage[all]{xy}
\usepackage{pgfplots, pgfplotstable, booktabs, colortbl, array,multirow}

\usepackage{framed}

\theoremstyle{plain}
\newtheorem{theorem}{Theorem}[section]
\newtheorem{lemma}[theorem]{Lemma}
\newtheorem{proposition}[theorem]{Proposition}

\theoremstyle{definition}
\newtheorem{definition}{Definition}[section]

\newtheorem{remark}{Remark}[section]

\DeclareMathOperator{\dv}{div}

\newcommand{\maketable}[1]{%
\pgfplotstabletypeset[
header=false,
font=\small,
clear infinite,
every head row/.style={before row=\hline,after row=\hline},
every last row/.style={after row=\hline},
every row no 4/.style={before row=\hline},
every row no 8/.style={before row=\hline},
every row no 12/.style={before row=\hline},
create on use/newcol/.style={
        create col/set list={NaN,0,NaN,NaN,NaN,1,NaN,NaN,NaN,2,NaN,NaN,}
},
columns={newcol,0,1,2,3,4},
columns/newcol/.style={column type/.add={|}{},column name={$r$}},
columns/0/.style={sci zerofill,column type/.add={|}{|},column name={$h^{-1}$}},
columns/1/.style={dec sep align={c|},sci,sci 10e,sci zerofill,precision=2,column type/.add={}{|},column name={$\|u_h-u\|_{L^2(\Omega)}$}},
columns/3/.style={dec sep align={c|},sci,sci 10e,sci zerofill,precision=2,column type/.add={}{|},column name={$\|\rho_h-\rho\|_{L^2(\Omega)}$}}, 
columns/2/.style={dec sep align={c|},fixed zerofill,precision=2,column type/.add={}{|},column name={Rate}},
columns/4/.style={dec sep align={c|},fixed zerofill,precision=2,column type/.add={}{|},column name={Rate}},
]
{#1}
}

\newcommand{\maketabledt}[1]{%
\pgfplotstabletypeset[
header=false,
font=\small,
clear infinite,
every head row/.style={before row=\hline,after row=\hline},
every last row/.style={after row=\hline},
every row no 4/.style={before row=\hline},
every row no 8/.style={before row=\hline},
columns={0,1,2,3,4},
columns/0/.style={sci zerofill,column type/.add={|}{|},column name={$\Delta t^{-1}$}},
columns/1/.style={dec sep align={c|},sci,sci 10e,sci zerofill,precision=2,column type/.add={}{|},column name={$\|u_h-u\|_{L^2(\Omega)}$}},
columns/3/.style={dec sep align={c|},sci,sci 10e,sci zerofill,precision=2,column type/.add={}{|},column name={$\|\rho_h-\rho\|_{L^2(\Omega)}$}}, 
columns/2/.style={dec sep align={c|},fixed zerofill,precision=2,column type/.add={}{|},column name={Rate}},
columns/4/.style={dec sep align={c|},fixed zerofill,precision=2,column type/.add={}{|},column name={Rate}}
]
{#1}
}

\begin{document}

\title{A Variational Finite Element Discretization of Compressible Flow}

\author{Evan S. Gawlik\thanks{\noindent Department of Mathematics,  University of Hawaii at Manoa, \href{egawlik@hawaii.edu}{egawlik@hawaii.edu}} \; and \; Fran\c{c}ois Gay-Balmaz\thanks{\noindent CNRS - LMD, Ecole Normale Sup\'erieure, \href{francois.gay-balmaz@lmd.ens.fr}{francois.gay-balmaz@lmd.ens.fr}}}

\date{}

\maketitle

\begin{abstract}
We present a finite element variational integrator for compressible flows. The numerical scheme is derived by discretizing, in a structure preserving way, the Lie group formulation of fluid dynamics on diffeomorphism groups and the associated variational principles.  Given a triangulation on the fluid domain, the discrete group of diffeomorphisms is defined as a certain subgroup of the group of linear isomorphisms of a finite element space of functions. In this setting, discrete vector fields correspond to a certain subspace of the Lie algebra of this group. This subspace is shown to be isomorphic to a Raviart-Thomas finite element space. The resulting finite element discretization corresponds to a weak form of the compressible fluid equation that doesn't seem to have been used in the finite element literature. It extends previous work done on incompressible flows and at the lowest order on compressible flows. We illustrate the conservation properties of the scheme with some numerical simulations.
\end{abstract}


\tableofcontents

\section{Introduction}

Numerical schemes that respect conservation laws and other geometric structures are of paramount importance in computational fluid dynamics, especially for problems relying on long time simulation. This is the case for geophysical fluid dynamics in the context of meteorological or climate prediction.

\medskip

Schemes that preserve the geometric structures underlying the equations they discretize are known as geometric integrators \cite{HaLuWa2006}. 
One efficient way to derive geometric integrators is to exploit the variational formulation of the continuous equations and to mimic this formulation at the spatial and/or temporal discrete level. For instance, in classical mechanics, a time discretization of the Lagrangian variational formulation permits the derivation of numerical schemes, called variational integrators, that are symplectic, exhibit good energy behavior, and inherit a discrete version of Noether's theorem which guarantees the exact preservation of momenta arising from symmetries, see \cite{MaWe2001}.

\medskip

Geometric variational integrators for fluid dynamics were first derived in \cite{PaMuToKaMaDe2010} for the Euler equations of a perfect fluid.  These integrators exploit the viewpoint of \cite{Ar1966} that fluid motions correspond to geodesics on the group of volume preserving diffeomorphisms of the fluid domain.  The spatially discretized Euler equations emerge from an application of this principle on a finite dimensional approximation of the diffeomorphism group.  The approach has been extended to various equations of incompressible fluid dynamics with advected quantities \cite{GaMuPaMaDe2011}, rotating and stratified fluids for atmospheric and oceanic dynamics \cite{DeGaGBZe2014}, reduced-order models of fluid flow \cite{LiMaHoToDe2015}, anelastic and pseudo-incompressible fluids on 2D irregular simplicial meshes \cite{BaGB2019b}, compressible fluids \cite{BaGB2019}, and compressible fluids on spheres \cite{BrBaBiGBML2019}.  In all of the aforementioned references, the schemes that result are low-order finite difference schemes.  

\medskip

It was suggested in \cite{LiMaHoToDe2015} that the variational discretization initiated in \cite{PaMuToKaMaDe2010} can be generalized by letting the discrete diffeomorphism group act on finite element spaces. Such an approach was developed in \cite{NaCo2018} in the context of the ideal fluid and thus allowed for a higher order version of the method as well as an error estimate.  For certain parameter choices, this high order method coincides with an $H(\dv)$-conforming finite element method studied in \cite{Guzman2016}.

\medskip

In the present paper we develop a finite element variational discretization of compressible fluid dynamics by exploiting the recent progresses made in \cite{BaGB2019} and \cite{NaCo2018}, based on the variational method initiated in \cite{PaMuToKaMaDe2010}. Roughly speaking, our approach is the following. Given a triangulation on the fluid domain $\Omega$, we consider the space $V_h^r\subset L^2(\Omega)$ of polynomials of degree $\leq r$ on each simplex and define the group of discrete diffeomorphisms as a certain subgroup $G_h^r$ of the general linear group $GL(V_h^r)$. The action of $G_h^r$ on $V_h^r$ is understood as a discrete version of the action by pull back on functions in $L^2(\Omega)$. As a consequence, the action of the Lie algebra $\mathfrak{g}^r_h$ on $V_h^r$ is understood as as discrete version of the derivation along vector fields. In a similar way with \cite{PaMuToKaMaDe2010} and \cite{BaGB2019}, this interpretation naturally leads one to consider a specific subspace of $\mathfrak{g}_h^r$ consisting of Lie algebra elements that actually represent discrete vector fields. We show that this subspace is isomorphic to a Raviart-Thomas finite element space. We also define a Lie algebra-to-vector fields map, that allows a systematic definition of the semidiscrete Lagrangian for any given continuous Lagrangian. The developed setting allows us to derive the finite element scheme by applying a discrete version of the Lie group variational formulation of compressible fluids. In particular the discretization corresponds to a weak form of the compressible fluid equation that doesn't seem to have been used in the finite element literature. An incompressible version of this expression of the weak form has been used in \cite{GaGB2019} for the incompressible fluid with variable density. The setting that we develop applies in general to 2D and 3D fluid models that can be written in Euler-Poincar\'e form. For instance it applies to the rotating shallow water equations.

\medskip

The plan of the paper is the following. In Section \ref{sec_review}, we review the variational Lie group formulation of both incompressible and compressible fluids, by recalling the Hamilton principle on diffeomorphism groups corresponding to the Lagrangian description, and the induced Euler-Poincar\'e variational principle corresponding to the Eulerian formulation. We also briefly indicate how this formulation has been previously used to derive variational integrators. In Section \ref{Sec_distributional_derivative} we consider the distributional derivative, deduce from it a discrete derivative acting on finite element spaces, and study its properties. In particular we show that these discrete derivatives are isomorphic to a Raviart-Thomas finite element space. In the lower order setting, the space of discrete derivatives recovers the spaces used in previous works, both for the incompressible \cite{PaMuToKaMaDe2010} and compressible \cite{BaGB2019} cases. In Section \ref{sec_Lie_to_vector}, we define a map that associates to any Lie algebra element of the discrete diffeomorphism group a vector field on the fluid domain. We call such a map a Lie algebra-to-vector fields map. It is needed to define in a general way the semidiscrete Lagrangian associated to a given continuous Lagrangian. We study its properties, which are used later to write down the numerical scheme. In Section \ref{sec_FEVA}, we derive the numerical scheme by using the Euler-Poincar\'e equations on the discrete diffeomorphism group associated to the chosen finite element space. As we will explain in detail, in a similar way with the approach initiated in  \cite{PaMuToKaMaDe2010}, a nonholonomic version of the Euler-Poincar\'e principle is used to constrain the dynamics to the space of discrete derivatives. We show that such a space must be a subspace of a Brezzi-Douglas-Marini finite element space. Finally, we illustrate the behavior of the resulting scheme in Section \ref{sec_examples}.

\section{Review of variational discretizations in fluid dynamics}\label{sec_review}

We begin by reviewing the variational formulation of ideal and compressible fluid flows and their variational discretization.

\subsection{Incompressible flow}\label{subsec_incomp}

\paragraph{The Continuous Setting.} As we mentioned in the introduction, solutions to the Euler equations of ideal fluid flow in a bounded domain $\Omega\subset \mathbb{R}^n$ with smooth boundary can be formally regarded as curves $\varphi : [0,T] \rightarrow \operatorname{Diff}_{\mathrm{vol}}(\Omega)$ that are critical for the Hamilton principle
\begin{equation}\label{action}
\delta \int_0^TL(\varphi, \partial_t\varphi){\rm d}t=0
\end{equation}
with respect to variations $\delta \varphi$ vanishing at the endpoints.  Here $\operatorname{Diff}_{\rm vol}(\Omega)$ is the group of volume preserving diffeomorphisms of $\Omega$ and $\varphi(t) : \Omega \rightarrow \Omega$ is the map sending the position $X$ of a fluid particle at time $0$ to its position $x=\varphi(t,X)$ at time $t$. The Lagrangian in \eqref{action} is given by the kinetic energy
\[
L(\varphi,\partial_t \varphi) = \int_\Omega \frac{1}{2}|\partial_t \varphi|^2 \, {\rm d}X,
\]
and is invariant under the right action of $\operatorname{Diff}_{\mathrm{vol}}(\Omega)$ on itself via composition, namely,
\[
L(\varphi \circ \psi,\partial_t (\varphi \circ \psi)) = L(\varphi,\partial_t \varphi), \quad \forall \psi \in \operatorname{Diff}_{\mathrm{vol}}(\Omega).
\]
This symmetry is often referred to as the particle relabelling symmetry.

\medskip

As a consequence of this symmetry, the variational principle \eqref{action} can be recast on the Lie algebra of $\operatorname{Diff}_{\mathrm{vol}}(\Omega)$, which is the space $\mathfrak{X}_{\dv}(\Omega)$ of divergence-free vector fields on $\Omega$ with vanishing normal component on $\partial\Omega$. Namely, one seeks a curve $u : [0,T] \rightarrow \mathfrak{X}_{\dv}(\Omega)$ satisfying the critical condition
\begin{equation} \label{variationalu}
\delta \int_0^T \ell(u) \, {\rm d}t = 0,
\end{equation}
subject to variations $\delta u$ of the form
\[
\delta u = \partial_t v + \pounds_u v,\quad \text{with } v : [0,T] \rightarrow \mathfrak{X}_{\dv}(\Omega) \text{ and } v(0)=v(T)=0,
\]
where 
\[
\ell(u) = \int_{\Omega} \frac{1}{2}|u|^2 \, {\rm d}x,
\]
and $\mathcal{L}_uv= [u,v]=u \cdot  \nabla v - v\cdot \nabla u$ is the Lie derivative of the vector field $v$ along the vector field $u$. This principle is obtained from the Hamilton principle \eqref{action} by using the relation $\partial_t\varphi= u \circ\varphi$ between the Lagrangian and Eulerian velocities  and by computing the constrained variations of $u$ induced by the free variations of $\varphi$. The conditions for criticality in \eqref{variationalu} read
\begin{equation}\label{Euler_equations}
\begin{aligned}
\partial_t u + u \cdot \nabla u &= -\nabla p, \\
\dv u &= 0, 
\end{aligned}
\end{equation}
where $p$ is a Lagrange multiplier enforcing the incompressibility constraint. The process just described for the Euler equations is valid in general for invariant Euler-Lagrange systems on arbitrary Lie groups and is known as Euler-Poincar\'e reduction. It plays an important role in this paper as it is used also at the discrete level to derive the numerical scheme. We refer to Appendix \ref{Appendix_A} for more details on the Euler-Poincar\'e principle and its application to incompressible flows.

\paragraph{The Semidiscrete Setting.} The variational principle recalled above has been used to derive structure-preserving discretizations of the incompressible Euler equations~\cite{PaMuToKaMaDe2010}, and various generalizations of it have been used to do the same for other equations in incompressible fluid dynamics~\cite{GaMuPaMaDe2011,DeGaGBZe2014}.  In these discretizations, the group $\operatorname{Diff}_{\mathrm{vol}}(\Omega)$ is approximated by a subgroup $\mathring{G}_h$ of the general linear group $GL(V_h)$ over a finite-dimensional vector space $V_h$, and extremizers of a time-discretized action functional are sought within $\mathring{G}_h$. More precisely, extremizers are sought within a subspace of the Lie algebra $\mathring{\mathfrak{g}}_h$ of $\mathring{G}_h$ after reducing by a symmetry and imposing nonholonomic constraints. This construction typically leads to schemes with good long-term conservation properties.

The use of a subgroup of the general linear group to approximate $\operatorname{Diff}_{\mathrm{vol}}(\Omega)$ is inspired by the fact that $\operatorname{Diff}_{\mathrm{vol}}(\Omega)$ acts linearly on the Lebesgue space $L^2(\Omega)$ from the right via the pullback,
\[
f \cdot \varphi = f \circ \varphi, \quad f \in L^2(\Omega), \, \varphi \in \operatorname{Diff}_{\mathrm{vol}}(\Omega).
\]
This action satisfies
\begin{equation} \label{const}
f \cdot \varphi = f, \quad \text{ if $f$ is constant},
\end{equation}
and it preserves the $L^2$-inner product $\langle f,g \rangle = \int_\Omega fg \, {\rm d}x$ thanks to volume-preservation:
\begin{equation} \label{inner}
\langle f \cdot \varphi, g \cdot \varphi \rangle = \langle f, g \rangle, \quad \forall f,g \in L^2(\Omega).
\end{equation}
In the discrete setting, this action is approximated by the (right) action of $GL(V_h)$ on $V_h$,
\begin{equation}\label{Gh_action}
f \cdot q = q^{-1} f, \quad f \in V_h, \, q \in GL(V_h).\footnote{Note that the representation of the group diffeomorphism by pull-back on functions is naturally a \textit{right} action ($f\mapsto f\circ \varphi$), whereas the group $GL(V_h)$ acts by matrix multiplication on the \textit{left} ($f\mapsto q f$). This explain the use of the inverse $q^{-1}$ on right hand side of \eqref{Gh_action}.}
\end{equation}
By imposing discretized versions of the properties \eqref{const} and \eqref{inner}, the group $\mathring{G}_h$ is taken equal to 
\begin{equation}\label{Gh}
\mathring{G}_h = \{ q \in GL(V_h) \mid q \textbf{1} = \textbf{1}, \, \langle qf , qg  \rangle = \langle f,g \rangle, \; \forall \, f,g \in V_h \},
\end{equation}
where $\textbf{1} \in V_h$ denotes a discrete representative of the constant function 1.

While the elements of $\mathring{G}_h$ are understood as discrete versions of volume preserving diffeomorphisms, elements in the Lie algebra
\begin{equation}\label{gh}
\mathring{\mathfrak{g}}_h =\{A\in L(V_h, V_h)\mid A\mathbf{1}=0, \, \langle A f , g \rangle + \langle f,A g \rangle, \; \forall \, f,g \in V_h\}
\end{equation}
of $\mathring{G}_h$ are understood as discrete volume preserving vector fields\footnote{Strictly speaking, only a subspace of this Lie algebra represents discrete vector fields, as we will see in detail later.}.  The linear (right) action of the Lie algebra element $A\in \mathring{\mathfrak{g}}_h$ on a discrete function $f\in V_h$, induced by the action \eqref{Gh_action} of $G_h$, is given by
\begin{equation}\label{representation}
f\cdot A = - Af\;\footnote{Note the minus sign due to \eqref{Gh_action}, which is consistent with the fact that $f\mapsto f\cdot A$ is a \textit{right} representation while $f\mapsto AF$ is a \textit{left} representation.}
\end{equation}
It is understood as the discrete derivative of $f$ in the direction $A$.

In early incarnations of this theory, $V_h$ is taken equal to $\mathbb{R}^N$, where $N$ is the number of elements in a triangulation of $\Omega$, and $\textbf{1} \in \mathbb{R}^N$ is the vector of all ones. In this case, we have $\langle F,G \rangle = F^\mathsf{T} \Theta G$, where $\Theta \in \mathbb{R}^{N \times N}$ is a diagonal matrix whose $i^{th}$ diagonal entry is the volume of the $i^{th}$ element of the triangulation.  Hence $\mathring{G}_h$ is simply the group of $\Theta$-orthogonal matrices with rows summing to 1.  

In more recent treatments, a finite element formulation has been adopted \cite{NaCo2018}.  Namely, $V_h$ is taken equal to a finite-dimensional subspace of $L^2(\Omega)$, with the inner product inherited from $L^2(\Omega)$, and $\mathbf{1}$ is simply the constant function $1$. This is the setting that we will develop to the compressible case in the present paper.

\subsection{Compressible flows}

\paragraph{The Continuous Setting.} The Lie group variational formulation recalled above generalizes to compressible flows as follows. For simplicity we consider here only the barotropic fluid, in which the internal energy is a function of the mass density only. The variational treatment of the general (or baroclinic) compressible fluid is similar, see Appendix \ref{Appendix_A}. Consider the group $\operatorname{Diff}(\Omega)$ of all, not necessarily volume preserving, diffeomorphisms of $\Omega$ and the Lagrangian
\begin{equation}\label{L_compressible}
L(\varphi,\partial_t \varphi, \varrho_0) = \int_\Omega \Big[\frac{1}{2}\varrho_0 |\partial_t \varphi|^2 - \varrho_0e(\varrho_0/ J\varphi)\Big]\, {\rm d}X.
\end{equation}
Here $\varrho_0$ is the mass density of the fluid in the reference configuration, $J\varphi$ is the Jacobian of the diffeomorphism $\varphi$, and $e$ is the specific internal energy of the fluid. The equations of evolution are found as before from the Hamilton principle
\begin{equation}\label{action_compressible}
\delta\int_0^TL(\varphi, \partial_t\varphi, \varrho_0) {\rm d}t=0,
\end{equation}
subject to arbitrary variations $\delta\varphi$ vanishing at the endpoints and where $\varrho_0$ is held fixed.

The main difference with the case of incompressible fluids recalled earlier is that the Lagrangian $L$ is not invariant under the configuration Lie group $\operatorname{Diff}(\Omega)$ but only under the subgroup $\operatorname{Diff}(\Omega)_{\varrho_0}\subset \operatorname{Diff}(\Omega)$ of diffeomorphisms that preserve $\varrho_0$, i.e., $\operatorname{Diff}(\Omega)_{\varrho_0}=\{\varphi \in \operatorname{Diff}(\Omega) \mid (\varrho_0\circ \varphi)J\varphi= \varrho_0\}$, namely we have
\begin{equation}\label{symmetry_compressible}
L(\varphi \circ \psi,\partial_t (\varphi \circ \psi), \varrho_0) = L(\varphi,\partial_t \varphi, \varrho_0), \quad \forall \psi \in \operatorname{Diff}(\Omega)_{\varrho_0},
\end{equation}
as it is easily seen from \eqref{L_compressible}.

\medskip

As a consequence of the symmetry \eqref{symmetry_compressible}, one can associate to $L$ the Lagrangian $\ell(u,\rho)$ in Eulerian form, as follows
\begin{equation}\label{Lagr_to_Euler}
L(\varphi,\partial_t \varphi, \varrho_0)= \ell(u,\rho),\quad \text{with}\;\;u = \partial_t\varphi\circ\varphi^{-1},\quad \rho= (\varrho_0\circ\varphi^{-1}) J \varphi^{-1}
\end{equation}
and
\begin{equation}\label{Lagra_comp}
\ell(u,\rho)= \int_\Omega \Big[\frac{1}{2}\rho|u|^2 - \rho e(\rho) \Big] {\rm d}x.
\end{equation}
From the relations \eqref{Lagr_to_Euler}, the Hamilton principle \eqref{action_compressible} induces the variational principle
\begin{equation} \label{variationalu_rho}
\delta \int_0^T \ell(u,\rho) \, {\rm d}t = 0,
\end{equation}
with respect to variations $\delta u$ and $\delta\rho$ of the form
\[
\delta u = \partial_t v + \pounds_u v,\quad \delta \rho = - \operatorname{div}(\rho v), \quad \text{with } v : [0,T] \rightarrow \mathfrak{X}(\Omega) \text{ and } v(0)=v(T)=0.
\]
Here $\mathfrak{X}(\Omega)$ denotes the Lie algebra of $\operatorname{Diff}(\Omega)$, which consists of vector fields on $\Omega$, with vanishing normal component on $\partial\Omega$. The conditions for criticality in \eqref{variationalu_rho} yield the balance of fluid momentum
\begin{equation}\label{Euler}
\rho(\partial_t u + u \cdot \nabla u) = -\nabla p, \quad \text{with}\quad p = \rho^2\frac{\partial e}{\partial \rho}
\end{equation}
while the relation $\rho=  (\varrho_0\circ\varphi^{-1}) J \varphi^{-1}$ yields the continuity equation
\[
\partial_t\rho+\operatorname{div}(\rho u) = 0.
\]
As in the case of incompressible flow, the process just described for the group $\operatorname{Diff}(\Omega)$ is a special instance of the process of Euler-Poincar\'e reduction. We refer to Appendix \ref{Appendix_A} for more details and to \S\ref{sec_comp} for the rotating fluid in a gravitational field.

\paragraph{The Semidiscrete Setting.} A low-order semidiscrete variational setting has been described in \cite{BaGB2019} that extends the work of \cite{PaMuToKaMaDe2010,GaMuPaMaDe2011,DeGaGBZe2014} to the compressible case, with a particular focus on the rotating shallow water equations. It is based on the compressible version of the discrete diffeomorphism group \eqref{Gh}, namely
\begin{equation}\label{Gh_comp}
G_h = \{ q \in GL(V_h) \mid q \textbf{1} = \textbf{1} \},
\end{equation}
whose Lie algebra is
\begin{equation}\label{gh_comp}
\mathfrak{g}_h = \{A\in L(V_h, V_h)\mid A\mathbf{1}=0\}.
\end{equation}
A nonholonomic constraint is imposed in \cite{BaGB2019} to distinguish elements of $\mathfrak{g}_h$ that actually represent discrete versions of vector fields.  In this paper, we will see how this idea generalizes to the higher order setting. The representation of $\mathfrak{g}_h$ on $V_h$ is given as before by $f\mapsto f\cdot A= - Af$ and is understood as a discrete version of the derivative in the direction $A$.

Notice that we denote by $\mathring{G}_h$ and $G_h$ the subgroups of $GL(V_h)$ when the finite element space $V_h$ is left unspecified, similarly for the corresponding Lie algebras $\mathring{\mathfrak{g}}_h$ and $\mathfrak{g}_h$. When it is chosen as the space $V_h^r$ of polynomials of degree $\leq r$ on each simplex, we use the notations $\mathring{G}_h^r$, $G_h^r$ for the groups and  $\mathring{\mathfrak{g}}_h^r$, $\mathfrak{g}_h^r$ for the Lie algebras.

\section{The distributional directional derivative and its properties}\label{Sec_distributional_derivative}

As we have recalled above, when using a subgroup of $GL(V_h)$ to discretize the diffeomorphism group, its Lie algebra $\mathfrak{g}_h$ contains the subspace of discrete vector fields. More precisely, as linear maps in $\mathfrak{g}_h\subset L(V_h,V_h)$, these discrete vector fields act as discrete derivations on $V_h$. Once a vector space $V_h$ is selected, it is thus natural to choose these discrete vector fields as distributional directional derivatives. In this section we recall this definition, study its properties and show that these derivations are isomorphic to a Raviart-Thomas finite element space.

\medskip

Let $\Omega$ be as before the domain of the fluid, assumed to be bounded with smooth boundary. We consider the Hilbert spaces
\begin{align*}
H(\dv,\Omega)&= \{ u \in L^2(\Omega)^n \mid \dv u \in L^2(\Omega)\}\\
H_0(\dv,\Omega)&= \{ u \in H(\dv,\Omega) \mid\left.u \cdot n\right|_{\partial\Omega} = 0\}\\
\mathring{H}(\dv,\Omega)&= \{u \in H(\dv,\Omega) \mid \dv u = 0, \, \left.u \cdot n\right|_{\partial\Omega} = 0 \}.
\end{align*}

\subsection{Definition and properties}

Let $\mathcal{T}_h$ be a triangulation of $\Omega$ having maximum element diameter $h$.  We assume that $\mathcal{T}_h$ belongs to a shape-regular, quasi-uniform family of triangulations of $\Omega$ parametrized by $h$.  That is, there exist positive constants $C_1$ and $C_2$ independent of $h$ such that
\[
\max_{K \in \mathcal{T}_h} \frac{h_K}{\rho_K} \le C_1, \text{ and } \max_{K \in \mathcal{T}_h} \frac{h}{h_K} \le C_2,
\]
where $h_K$ and $\rho_K$ denote the diameter and inradius of a simplex $K$.
For $r \ge 0$ an integer, we consider the subspace of $L^2(\Omega)$
\begin{equation} \label{Vh}
V_h^r = \{ f \in L^2(\Omega) \mid \left.f\right|_K \in P_r(K),\;   \forall K \in \mathcal{T}_h \},
\end{equation}
where $P_r(K)$ denotes the space of polynomials of degree $\le r$ on a simplex $K$.

\medskip

\begin{definition} Given $u\in H(\operatorname{div}, \Omega)$, the \textbf{\textit{distributional derivative in the direction $u$}} is the linear map $\nabla_u^{\rm dist} : L^2(\Omega) \rightarrow C^\infty_0(\Omega)'$ defined by
\begin{equation}\label{def_nabla_dist}
\int_\Omega (\nabla_u^{\rm dist} f ) g \,{\rm  d}x = -\int_\Omega f \operatorname{div}(gu)  \, {\rm d}x, \quad \forall g \in C^\infty_0(\Omega).
\end{equation}
\end{definition}

\medskip

When a triangulation $\mathcal{T}_h$ is fixed, $f\in V_h^r$, and $u\in   H_0(\operatorname{div}, \Omega)\cap L^p(\Omega)^n$, $p>2$, the distributional directional derivative \eqref{def_nabla_dist} can be rewritten as
\begin{equation}\label{rewriting_nabla_dist}
\begin{aligned}
\int_\Omega ( \nabla_u^{\rm dist} f ) g \, {\rm d}x&= \sum_{K \in \mathcal{T}_h} \int_K (\nabla_u f) g \, {\rm d}x - \sum_{K\in \mathcal{T}_h}\int_{\partial K} (u\cdot n) fg {\rm d}s\\
&= \sum_{K \in \mathcal{T}_h} \int_K (\nabla_u f) g \, {\rm d}x - \sum_{e \in \mathcal{E}_h^0} \int_e u \cdot \llbracket f \rrbracket g \,{\rm  d}s,
\end{aligned}
\end{equation}
for all $g\in C^\infty_0(\Omega)$. Here, $\nabla _uf= u \cdot\nabla f$ denotes the derivative of $f$ along $u$, $\mathcal{E}_h^0$ denotes the set of interior $(n-1)$-simplices in $\mathcal{T}_h$ (edges in two dimensions), and $\llbracket f \rrbracket$ is defined by
\[
\llbracket f \rrbracket := f_1 n_1 + f_2 n_2, \quad \text{on}\quad e = K_1 \cap K_2 \in \mathcal{E}_h^0,
\]
with $f_i := \left.f\right|_{K_i}$, $n_1$ the normal vector to $e$ pointing from $K_1$ to $K_2$, and similarly for $n_2$.

\medskip

Note that there are some subtleties that arise when looking at traces on subsets of the boundary if the trace is a distribution, which explains why we need to take $u\in   H_0(\operatorname{div}, \Omega)\cap L^p(\Omega)^n$, for some $p>2$ when passing to the second line in \eqref{rewriting_nabla_dist}.  The trace of a vector field $u\in H(\dv,K)$ on $\partial K$ satisfies $u \cdot n \in H^{-1/2}(\partial K) = H^{1/2}_0(\partial K)' = H^{1/2}(\partial K)'$, but the trace of $u$ on $e \subset \partial K$ satisfies $u \cdot n \in H^{1/2}_{00}(e)'$, where $H^{1/2}_{00}(e)$ defined by
\[
H^{1/2}_{00}(e)=\{ g \in H^{1/2}(e) \mid \text{ the zero-extension of $g$ to $\partial K$ belongs to $H^{1/2}(\partial K)$} \} \subsetneq H^{1/2}_0(e),
\]
see, e.g., \cite{Ba2012}. So for $u\in H(\dv,\Omega)$ and smooth $g$, $\int_{\partial K} (u \cdot n) g \,{\rm  d}s$ is always well-defined, but $\int_e (u \cdot n) g \,{\rm  d}s$ need not be; some extra regularity for $u$ is required to make it well-defined.

\medskip

\begin{definition} Given $A\in L(V_h^r, V_h^r)$ and $u\in H_0(\operatorname{div},\Omega)\cap L^p(\Omega)^n$, $p>2$, we say that \textbf{\textit{$A$ approximates $-u$ in $V^r_h$}}\footnote{The fact that $A$ approximates $-u$ and not $u$ is consistent with the fact that $f\mapsto Af$ is a left Lie algebra action whereas the derivative $f\mapsto \nabla_u f$ is a right Lie algebra action.} if whenever $f \in L^2(\Omega)$ and $f_h \in V_h^r$ is a sequence satisfying $\|f - f_h \|_{L^2(\Omega)} \rightarrow 0$, we have
\begin{equation} \label{discreteLiedef}
\langle A f_h - \nabla_u^{\rm dist} f, g \rangle \rightarrow 0, \quad \forall g \in C^\infty_0(\Omega).
\end{equation}
In other words, we require that $A$ is a consistent approximation of $\nabla_u^{\rm dist}$ in $V^r_h$.
\end{definition}
\medskip
Note that the above definition abuses notation slightly; we are really dealing with a sequence of $A$'s parametrized by $h$.
\medskip

\begin{proposition}\label{A_u_approximation} Given $u\in H_0(\dv, \Omega) \cap L^p(\Omega)^n$ and $r\geq 0$ an integer, a consistent approximation of $\nabla_u^{\rm dist}$ in $V_h^r$ is obtained by setting $A=A_u\in L(V_h^r, V_h^r)$ defined by
\begin{equation} \label{consistentapprox}
\langle A_u f, g \rangle := \sum_{K \in \mathcal{T}_h} \int_K (\nabla_u f) g \, {\rm d}x - \sum_{e \in \mathcal{E}_h^0} \int_e u \cdot \llbracket f \rrbracket \{g\} \, {\rm d}s, \quad \forall f,g \in V_h^r,
\end{equation}
where $\{g\} := \frac{1}{2}(g_1 + g_2)$ on $e=K_1\cap K_2$.

Moreover, if $r\geq 1$, $p=\infty$, and $A\in L(V_h^r, V_h^r)$ is any other operator that approximates $u$ in $V^r_h$, then $A$ must be close to $A_u$ in the following sense: if $f \in L^2(\Omega)$, $g \in C^\infty_0(\Omega)$, and if $f_h,g_h \in V_h^r$ satisfy $\|f_h-f\|_{L^2(\Omega)} \rightarrow 0$ and $h^{-1}\|g-g_h\|_{L^2(\Omega)} + \left(\sum_{K \in \mathcal{T}_h} |g-g_h|_{H^1(K)}^2\right)^{1/2} \rightarrow 0$, then
\[
\langle (A-A_u) f_h, g_h \rangle \rightarrow 0,
\]
provided that $\|Af_h\|_{L^2(\Omega)} \le C(u,f) h^{-1}$ for some constant $C(u,f)$.
\end{proposition}

As we will see in \S\ref{sec:loworder}, for $r=0$, the definition of $A_u$ in \eqref{consistentapprox} recovers the one used in \cite{BaGB2019}. There, a different definition of ``$A$ approximates $-u$'' than \eqref{discreteLiedef} was considered for the particular case $r=0$. When this definition is used, analogous statements of both parts of Proposition \ref{A_u_approximation} hold for $r=0$, under an additional assumption of the family of meshes~\cite[Lemma 2.2]{BaGB2019}.
\begin{proof} The operator $A_u$ is a consistent approximation of $\nabla_u^{\rm dist}$, since for all $g\in C^\infty_0(\Omega)$
\begin{align*}
\langle A_u f_h -\nabla_u^{\rm dist} f, g \rangle
&= \langle \nabla_u^{\rm dist} (f_h - f), g \rangle \\
&= -\int_{\Omega} (f_h-f)\operatorname{div}(gu){\rm d}x \\
&\le \|f_h - f\|_{L^2(\Omega)} \|u \cdot \nabla g + g \dv u \|_{L^2(\Omega)} \;\rightarrow \; 0.
\end{align*}
For the second part, we note that
\begin{align*}
\langle A f_h, g_h \rangle - \langle A_u f_h, g_h \rangle &= \langle A f_h - A_u f_h, g \rangle + \langle A f_h, g_h-g \rangle - \langle A_u f_h, g_h-g \rangle \\
&= \langle A f_h - \nabla_u^{\rm dist}  f_h, g \rangle + \langle A f_h, g_h - g \rangle - \langle A_u f_h, g_h-g \rangle \\
&= \langle A f_h - \nabla_u^{\rm dist}  f, g \rangle + \langle \nabla_u^{\rm dist} (f-f_h), g \rangle  + \langle A f_h, g_h - g \rangle - \langle A_u f_h, g_h-g \rangle \\
&\le |\langle A f_h -\nabla_u^{\rm dist}  f, g \rangle| + \|f-f_h\|_{L^2(\Omega)} \|\dv(ug)\|_{L^2(\Omega)} 
\\&\quad +  \|Af_h\|_{L^2(\Omega)} \|g_h-g\|_{L^2(\Omega)} + | \langle A_u f_h, g_h-g  \rangle|.
\end{align*}
By assumption, the first three terms above tend to zero as $h \rightarrow 0$.  
The last term satisfies
\begin{align*}
\langle A_u f_h, g_h - g \rangle
&= \langle A_u f_h, g_h \rangle - \langle \nabla_u^{\rm dist} f_h, g \rangle \\
&= \sum_{K \in \mathcal{T}_h} \int_K (\nabla_u f_h) (g_h-g) \, {\rm d}x - \sum_{e \in \mathcal{E}_h^0} \int_e u \cdot \llbracket f_h \rrbracket \{g_h-g\} \, {\rm d}s
\end{align*}
since $\{g\} = g$ on each $e \in \mathcal{E}_h^0$.  
To analyze these integrals, we make use of the inverse estimate~\cite{ErGu2004}
\[
\|f_h\|_{H^1(K)} \le C h_K^{-1} \|f_h\|_{L^2(K)}, \quad \forall f_h \in V_h^r, \, \forall K \in \mathcal{T}_h,
\]
and the trace inequality~\cite{Ar1982}
\[
\|f\|_{L^2(\partial K)} \le C\left( h_K^{-1/2} \|f\|_{L^2(K)} + h_K^{1/2} |f|_{H^1(K)} \right), \quad \forall f \in H^1(K), \, \forall K \in \mathcal{T}_h.
\]
Using the inverse estimate, we see that
\begin{align*}
\left| \int_K (\nabla_u f_h) (g_h-g) \, {\rm d}x \right|
&\le \|u\|_{L^\infty(\Omega)} |f_h|_{H^1(K)} \|g_h-g\|_{L^2(K)} \\
&\le Ch_K^{-1} \|u\|_{L^\infty(\Omega)} \|f_h\|_{L^2(K)} \|g_h-g\|_{L^2(K)}.
\end{align*}
Using the trace inequality and the inverse estimate, we see also that
\begin{align*}
&\left| \int_e u \cdot \llbracket f_h \rrbracket \{g_h-g\} \, {\rm d}s \right| \\
& \le \frac{1}{2} \|u\|_{L^\infty(\Omega)} \left( \|f_{h1}\|_{L^2(e)} + \|f_{h2}\|_{L^2(e)}  \right) \left( \|g_{h1}-g_1\|_{L^2(e)} + \|g_{h2}-g_2\|_{L^2(e)}  \right) \\
&\le C \|u\|_{L^\infty(\Omega)} \left( h_{K_1}^{-1/2} \|f_h\|_{L^2(K_1)} + h_{K_2}^{-1/2} \|f_h\|_{L^2(K_2)} \right) \\
&\quad \times \left( h_{K_1}^{-1/2} \|g_h-g\|_{L^2(K_1)} + h_{K_1}^{1/2} |g_h-g|_{H^1(K_1)} + h_{K_2}^{-1/2} \|g_h-g\|_{L^2(K_2)} + h_{K_2}^{1/2} |g_h-g|_{H^1(K_2)} \right),
\end{align*}
where $K_1,K_2 \in \mathcal{T}_h$ are such that $e = K_1 \cap K_2$, $f_{hi} = \left. f_h \right|_{K_i}$, $g_{hi} = \left. g_h \right|_{K_i}$, and $g_i = \left. g \right|_{K_i}$.
Summing over all $K \in \mathcal{T}_h$ and all $e \in \mathcal{E}_h^0$, and using the quasi-uniformity of $\mathcal{T}_h$, we get
\begin{align*}
|\langle A_u f_h, g_h - g \rangle|
&\le C \|u\|_{L^\infty(\Omega)} \|f_h\|_{L^2(\Omega)} \Big( h^{-1} \|g_h-g\|_{L^2(\Omega)} + \Big(\sum_{K \in \mathcal{T}_h} |g-g_h|_{H^1(K)}^2\Big)^{1/2} \Big) \rightarrow 0.
\end{align*}
\end{proof}

\medskip

Note that formula \eqref{consistentapprox} for $A_u$ is obtained from formula \eqref{rewriting_nabla_dist}, valid for $f\in V^r_h$ and $g\in C^\infty_0(\Omega)$, by rewriting it for the case where $g\in V^r_h$ and choosing to replace $g\rightarrow \{g\}$ in the second term in \eqref{rewriting_nabla_dist}.

\medskip

\begin{proposition}\label{A_u_properties} For all $u\in H_0(\operatorname{div}, \Omega)\cap L^p(\Omega)$, $p>2$, we have
\begin{equation}\label{first_statement}
A_u \mathbf{1}=0\quad\text{and}\quad \langle A_u f, g \rangle + \langle f, A_ug \rangle + \langle f, (\dv u)g \rangle = 0, \quad \forall f,g \in V_h^r.
\end{equation}
Hence, if $u\in \mathring{H}(\operatorname{div}, \Omega)$, then
\begin{equation}\label{second_statement}
A_u \mathbf{1}=0\quad\text{and}\quad \langle A_u f,g \rangle + \langle f, A_u g \rangle = 0, \quad\forall f,g \in V_h^r.
\end{equation}
\end{proposition}
\begin{proof} The first property $A_u\mathbf{1}=0$ follows trivially from the expression \eqref{consistentapprox} since $\nabla_u 1=0$ and $\llbracket 1 \rrbracket=0$.
We now prove the second equality. Using \eqref{consistentapprox}, we compute
\begin{align*}
&\langle A_u f, g \rangle + \langle f, A_ug \rangle\\
&=\sum_{K \in \mathcal{T}_h} \int_K \left((\nabla_u f) g + ( \nabla_ug) f \right)\, {\rm d}x - \sum_{e \in \mathcal{E}_h^0} \int_e u \cdot \llbracket f \rrbracket \{g\} \, {\rm d}s - \sum_{e \in \mathcal{E}_h^0} \int_e u \cdot \llbracket g \rrbracket \{f\} \, {\rm d}s\\
&=\sum_{K \in \mathcal{T}_h} \int_K \left(\operatorname{div}(fg u) - fg \operatorname{div} u \right)\, {\rm d}x - \sum_{e \in \mathcal{E}_h^0} \int_e u \cdot\left( \llbracket g \rrbracket \{f\}+\llbracket f \rrbracket \{g\}\right) \, {\rm d}s \\
&=- \sum_{K \in \mathcal{T}_h} \int_K  fg \operatorname{div} u \, {\rm d}x + \sum_{K \in \mathcal{T}_h}  \sum_{e\in K} \int_e f_e^Kg_e^K u \cdot n ^K_e \,{\rm d}s\\
&\qquad  -\sum_{e \in \mathcal{E}_h^0}  \int_e u  \cdot \Big(( \sum_{K\ni e} g_e^K n_e^K) \frac{1}{2}\sum_{K\ni e} f_e^K + ( \sum_{K\ni e} f_e^K n_e^K) \frac{1}{2}\sum_{K\ni e} g_e^K\Big) \,{\rm d}s\\
&=- \sum_{K \in \mathcal{T}_h} \int_K  fg \operatorname{div} u \, {\rm d}x \\
&\qquad  +\sum_{e \in \mathcal{E}_h^0}  \int_e u  \cdot \Big( \sum_{K\ni e} f_e^K g_e^K n_e^K- ( \sum_{K\ni e} g_e^K n_e^K) \frac{1}{2}\sum_{K\ni e} f_e^K - ( \sum_{K\ni e} f_e^K n_e^K) \frac{1}{2}\sum_{K\ni e} g_e^K\Big) \,{\rm d}s,
\end{align*}
where the sum $\sum_{K\ni e}$ is just the sum over the two simplices neighboring $e$. We have denoted by $n^K_e$ the unit normal vector to $e$ pointing outside $K$ and by $f^K_e, g^K_e$ the values of $f|_K$ and $g|_K$ on the hyperface $e$. Now for each $e\in \mathcal{E}^0_h$ and denoting simply by $1$ and $2$ the simplices $K_1,K_2$ with $K_1\cap K_2=e$, the last term in parenthesis can be written as
\begin{align*}
&f_e^1g_e^1n_e^1+ f_e^2g_e^2n_e^2- (g_e^1n_e^1+g_e^2n_e^2) \frac{1}{2}(f_e^1+f_e^2) - (f_e^1n_e^1+f_e^2n_e^2) \frac{1}{2}(g_e^1+g_e^2)\\
&= 0 + \frac{1}{2}\left(g_e^1 n_e^1 f_e^2 + g_e^2n_e^2f_e^1+ f_e^1n_e^1g_e^2+ f_e^2n_e^2g_e^1 \right)=0,
\end{align*}
since $n_e^2=-n_e^1$.\end{proof}

\medskip

From the previous result, we get a well-defined linear map
\begin{equation}\label{map_A}
\mathsf{A}: H_0(\dv,\Omega)\cap L^p(\Omega)^n\rightarrow \mathfrak{g}_h^r\subset L(V_h^r, V_h^r),\quad u\mapsto \mathsf{A}(u)=A_u,\quad p>2,
\end{equation}
with values in the Lie algebra $\mathfrak{g}_h^r=\{A\in L(V_h^r,V_h^r) \mid A\mathbf{1}=0\}$ of $G_h$. In the divergence free case, it restricts to
\[
\mathsf{A}: \mathring{H}(\dv,\Omega)\cap L^p(\Omega)^n\rightarrow \mathring{\mathfrak{g}}_h^r\subset L(V_h^r, V_h^r),
\]
with $\mathring{\mathfrak{g}}_h^r=\{A\in L(V_h^r,V_h^r) \mid A\mathbf{1}=0,\; \langle A f,g \rangle + \langle f, A g \rangle = 0,\;\forall f,g\in V^r_h\}$ the Lie algebra of $\mathring{G}_h$; see \eqref{Gh}.

\subsection{Relation with Raviart-Thomas finite element spaces}

\begin{definition} For $r\geq0$ an integer, we define the subspace $S_h^r \subset \mathfrak{g}_h^r \subset L(V_h^r,V_h^r)$ as
\[
S_h^r:=\operatorname{Im}\mathsf{A}=  \{ A_u \in L(V_h^r,V_h^r) \mid u \in H_0(\operatorname{div}, \Omega)\}.
\]
\end{definition}

\begin{proposition}\label{important_prop} Let $r\geq0$ be an integer. The space $S_h^r\subset \mathfrak{g}^r_h$ is isomorphic to the Raviart-Thomas space of order $2r$
\[
RT_{2r}(\mathcal{T}_h) = \left\{ u \in H_0(\dv,\Omega) \mid  \left.u\right|_K \in (P_{2r}(K))^n + x P_{2r}(K), \, \forall K \in \mathcal{T}_h \right\}.
\]
An isomorphism is given by $u\in RT_{2r}(\mathcal{T}_h)\mapsto A_u \in S_h^r$.

Its inverse is given by
\begin{equation}\label{inverse_isomorphism}
A \in S^r_h \mapsto u= \sum_{K\in \mathcal{T}_h}  \sum_{\alpha} \boldsymbol{\phi}_K^\alpha \sum_j \langle A f^{\alpha,j}_K, g^{\alpha,j}_K\rangle  + \sum_{e\in  \mathcal{E}_h^0} \sum_{\beta} \boldsymbol{\phi}_e^\beta \sum_j \langle A f^{\beta,j}_e, g^{\beta,j}_e\rangle  \in RT_{2r}(\mathcal{T}_h).
\end{equation}
In this formula:
\begin{itemize}
\item $(f^{\alpha,j}_K, g^{\alpha,j}_K), (f^{\beta,j}_e, g^{\beta,j}_e)\in V_h^r\times V_h^r$ are such that the images of $\sum_j (f^{\alpha,j}_K, g^{\alpha,j}_K) \in V_h^r \otimes V_h^r$ and $\sum_j (f^{\beta,j}_e, g^{\beta,j}_e) \in V_h^r \otimes V_h^r$ under the map
\begin{equation}\label{surjective_map}
V_h^r\otimes V_h^r\rightarrow RT_{2r}(\mathcal{T}_h)^*, \quad f\otimes g  \longmapsto \sum_{K\in \mathcal{T}_h} (\nabla f \, g)|_K + \sum_{e \in \mathcal{T}_h} \llbracket f\rrbracket_e \{g\}_e 
\end{equation}
are $(\mathbf{p}^\alpha_K, 0)$ and $(0, p^\beta_e)$, respectively, which is a basis of the dual space $RT_{2r}(\mathcal{T}_h)^*$ adapted to the decomposition
\[
RT_{2r}(\mathcal{T}_h)^*=\sum_K P_{2r-1}(K)^n \oplus \sum_{e\in K} P_{2r}(e);
\]
\item $\boldsymbol{\phi}_K^\alpha$, $\boldsymbol{\phi}_e^\beta$ is a basis of $RT_{2r}(\mathcal{T}_h)$ dual to the basis $\mathbf{p}^\alpha_K$ and $p^\beta_e$ of $RT_{2r}(\mathcal{T}_h)^*$.
\end{itemize}
\end{proposition}
\begin{proof} Let us consider the linear map $\mathsf{A}: H_0(\dv,\Omega)\rightarrow  L(V_h^r, V_h^r)$ defined in \eqref{map_A}. From a general result of linear algebra, we have $\operatorname{dim}(\operatorname{Im}\mathsf{A})=\operatorname{dim}(\operatorname{Im}\mathsf{A}^*)$, where $\mathsf{A}^*: L(V_h^r, V_h^r)^*\rightarrow H_0(\dv,\Omega)^*$ is the adjoint to $\mathsf{A}$. We have
\[
\operatorname{Im}\mathsf{A}^*= \left\{ \sum_{i=1}^N c_i \sigma_{f_i g_i} \in  H_0(\dv,\Omega)^*\;\Big|\;  N \in \mathbb{N},\, f_i,g_i \in V_h^r,\,  c_i \in \mathbb{R}, \, i=1,2,\dots,N  \right\},
\]
where the linear form $\sigma_{fg}:= \mathsf{A}^*(f\otimes g) \colon H_0(\dv,\Omega) \rightarrow \mathbb{R}$, is given by $\sigma_{fg}(u)=\langle f, A_u g \rangle$.

Now, the space $\operatorname{Im}\mathsf{A}^*$
is spanned by functionals of the form
\begin{equation} \label{functional1}
u \mapsto \int_e (u \cdot n) pq \, {\rm d}s, \quad p,q \in P_r(e), \, e \in \mathcal{E}_h^0,
\end{equation}
and
\begin{equation} \label{functional2}
u \mapsto \int_K (u \cdot \nabla q) p \, {\rm d}x \quad \, p,q \in P_r(K), \, K \in \mathcal{T}_h.
\end{equation}
This can be seen by choosing appropriate $f$ and $g$ in $\sigma_{fg}$ and using the definition \eqref{consistentapprox} of $A_u$.  Indeed, if we choose two adjacent simplices $K_1$ and $K_2$  and set $\left. f \right|_{K_1} = p \in P_r(K_1)$, $\left. g \right|_{K_2} = -2q \in P_r(K_2)$, $\left.f\right|_{\Omega \setminus K_1} = 0$, and $\left.g\right|_{\Omega \setminus K_2} = 0$, we get $\langle A_u f, g \rangle = \int_e (u \cdot n) p q \, {\rm d}s$ with $e=K_1 \cap K_2$.  Likewise, if we choose a simplex $K$ and set $\left. f \right|_K = q \in P_r(K)$, $\left. g \right|_K = p \in P_r(K)$, and $\left.f\right|_{\Omega \setminus K} = \left.g\right|_{\Omega \setminus K} = 0$, we get $\langle A_u f, g \rangle = \int_K (u \cdot \nabla q) p \, {\rm d}x + \frac{1}{2}\int_{\partial K} (u \cdot n) p q \, {\rm d}s$.  Taking appropriate linear combinations yields the functionals~(\ref{functional1}-\ref{functional2}).

Now observe that the functionals~(\ref{functional1}-\ref{functional2}) span the same space (see Lemmas~\ref{lemma:PrPr} and~\ref{lemma:PrgradPr} in Appendix~\ref{Appendix_C}) as the functionals
\[
u \mapsto \int_e (u \cdot n) p \, {\rm d}s, \quad p \in P_{2r}(e), \, e \in \mathcal{E}_h^0,
\]
and
\[
u \mapsto \int_K u \cdot p \, {\rm d}x \quad \, p \in P_{2r-1}(K)^n, \, K \in \mathcal{T}_h.
\] 
These functionals are well-known \cite{BrFo1991}: they are a basis for the dual of 
\[
RT_{2r}(\mathcal{T}_h) = \{ u \in H(\dv,\Omega) \mid \left.u \cdot n \right|_{\partial\Omega} = 0, \, \left.u\right|_K \in (P_{2r}(K))^n + x P_{2r}(K), \, \forall K \in \mathcal{T}_h \},
\]
often referred to as the ``degrees of freedom'' for $RT_{2r}(\mathcal{T}_h)$.
 
We thus have proven that $\operatorname{Im}\mathsf{A}^*$ is isomorphic to $RT_{2r}(\mathcal{T}_h) ^*$, and hence $S_h^r= \operatorname{Im}\mathsf{A}$ is isomorphic to $RT_{2r}(\mathcal{T}_h)$, all these spaces having the same dimensions. Now, let us consider the linear map $u\in RT_{2r}(\mathcal{T}_h)\rightarrow \mathsf{A}(u)=A_u\in S^r_h$. Since its kernel is zero, the map is an isomorphism.

\medskip

Consider now a basis ${\bf p}^\alpha_K$, $p^\beta_e$ of the dual space $RT_{2r}(\mathcal{T}_h)^*$ identified with $\sum_{K\in \mathcal{T}_h} P_{2r-1}(K)^n \oplus \sum_{e\in\mathcal{T}_h}P_{2r}(e)$, i.e., the collection $\{{\bf p}^\alpha_K\}$ is a basis of $\sum_{K\in \mathcal{E}_h} P_{2r-1}(K)^n$ and the collection $p^\beta_e$ is a basis of $\sum_{e\in\mathcal{T}_h}P_{2r}(e)$. There is a dual basis $\boldsymbol{\phi}_K^\alpha$, $\boldsymbol{\phi}_e^\beta$ of $RT_{2r}(\mathcal{T}_h)$ such that
\[
\langle ({\bf p}^\alpha_K, 0), \boldsymbol{\phi}_{K'}^{\alpha'}\rangle=\delta_{\alpha\alpha'}\delta_{KK'}\qquad \langle ({\bf p}^\alpha_K, 0), \boldsymbol{\phi}_{e}^{\beta}\rangle=0
\]
\[
\langle (0, p^\beta_e), \boldsymbol{\phi}_{e'}^{\beta'}\rangle=\delta_{\beta\beta'}\delta_{ee'}\qquad \langle (0, p^\beta_e), \boldsymbol{\phi}_{K}^{\alpha}\rangle=0,
\]
where
\[
\langle (\mathbf{p},p), u \rangle = \sum_K \int_K u \cdot \mathbf{p}\, {\rm d}x + \sum_e \int_e (u \cdot n)p\, {\rm d}s
\]
is the duality pairing between $RT_{2r}(\mathcal{T}_h)^*$ and $RT_{2r}(\mathcal{T}_h)$.

Choosing $f^{\alpha,j}_K, g^{\alpha,j}_K\in V_h^r$ and $f^{\beta,j}_e, g^{\beta,j}_e\in V_h^r$ such that
\[
\sum_j \nabla f^{\alpha,j}_K g^{\alpha,j}_K|_{K'}= \mathbf{p}^\alpha_K\delta_{KK'}, \qquad \sum_j \llbracket f^{\alpha,j}_K\rrbracket_e \{g^{\alpha,j}_K\}_e=0,
\]
\[
\sum_j \nabla f^{\beta,j}_e g^{\beta,j}_e|_K=0, \qquad  \sum_j \llbracket f^{\beta,j}_e  \rrbracket_{e'} \{g^{\beta,j}_e\}_{e'}=p_e^\beta\delta_{ee'},
\]
we have that
\[
\sum_j \langle A_u f^{\alpha,j}_K, g^{\alpha,j}_K\rangle \quad \text{and}\quad \sum_j \langle A_u f^{\beta,j}_e, g^{\beta,j}_e\rangle
\]
are exactly the degrees of freedom of $u$ relative to the basis ${\bf p}^\alpha_K$, $p^\beta_e$. Therefore, $u$ is expressed as
\[
u= \sum_K  \sum_\alpha \boldsymbol{\phi}_K^\alpha \sum_j \langle A_u f^{\alpha,j}_K, g^{\alpha,j}_K\rangle  + \sum_{e\in K} \sum_\beta \boldsymbol{\phi}_e^\beta \sum_j \langle A_u f^{\beta,j}_e, g^{\beta,j}_e\rangle 
\]
as desired.
\end{proof}

\medskip

Note that the map \eqref{surjective_map}, with $V_h^r\otimes V_h^r$ identified with $L(V_h^*, V_h^r)$ can be identified with the composition $I^* \circ \mathsf{A}^*$, where $I^*$ is the dual map to the inclusion $I: RT_{2r}(\mathcal{T}_h)\rightarrow H_0(\dv, \Omega)$. This map is surjective, from the preceding result.

\medskip

\begin{proposition} The kernel of the map $u \in H_0(\dv,\Omega)\cap L^p(\Omega)^n \mapsto \mathsf{A}(u)=A_u \in L(V_h^r, V_h^r)$, $p > 2$, is
\[
\operatorname{ker}\mathsf{A}= \{u \in H_0(\dv,\Omega)\cap L^p(\Omega) \mid\Pi_{2r}(u)=0\}= \operatorname{ker}\Pi_{2r},
\]
where $\Pi_{2r}: H_0(\dv,\Omega) \cap L^p(\Omega)^n\rightarrow RT_{2r}(\mathcal{T}_h)$ is the global interpolation operator defined by $\Pi_{2r}(v)|_K:= \Pi_{2r}^K(v|_K)$, with $\Pi_{2r}^K: H(\dv, K) \cap L^p(K)^n\rightarrow RT_{2r}(K)$ defined by the two conditions
\[
\int_e \big((u-\Pi_{2r}^K u)\cdot n\big) p\, {\rm d} s=0,\quad\text{for all $p\in P_{2r}(e)$, for all $e\in K$}
\]
and
\[
\int_K (u- \Pi_{2r}^K u)\cdot p\, {\rm d} x=0,\quad\text{for all $p\in P_{2r-1}(K)^n$}.
\]
\end{proposition}
\begin{proof} For $u \in H_0(\dv,\Omega) \cap L^p(\Omega)^n$ we have $A_u=0$ if and only if $\langle A_u f, g\rangle=0$ for all $f, g\in V_h^r$. As we just commented above, the map \eqref{surjective_map} is surjective, hence from \eqref{consistentapprox} we see that $\langle A_u f, g\rangle=0$ for all $f, g\in V_h^r$ holds if and only if 
\[
\int_e (u \cdot n) p \, {\rm d}s=0, \quad \text{for all}\;\; p \in P_{2r}(e), \, e \in \mathcal{E}_h^0,
\]
and
\[
\int_K u \cdot p \, {\rm d}x=0, \quad \text{for all}\;\; p \in P_{2r-1}(K)^n, \, K \in \mathcal{T}_h.
\] 
This holds if and only if $\Pi_{2r}(u)=0$.
\end{proof}

\medskip

In particular for $u\in H_0(\dv,\Omega)\cap L^p(\Omega)$, $p>2$, there exists a unique $\bar u \in RT_{2r}(\mathcal{T}_h)$ such that $A_{\bar u}= A_u$. It is given by $\bar u = \Pi_{2r}(u)$.

\subsection{The lowest-order setting} \label{sec:loworder}

We now investigate the setting in which $r=0$ in order to connect with the previous works \cite{PaMuToKaMaDe2010} and \cite{BaGB2019} for both the incompressible and compressible cases. Enumerate the elements of $\mathcal{T}_h$ arbitrarily from 1 to $N$, and let $\{\psi_i\}_i$ be the orthogonal basis for $V_h^0$ given by
\[
\psi_i(x) = 
\begin{cases}
1, &\mbox{ if } x \in K_i, \\
0, &\mbox{ otherwise, }
\end{cases}
\]
where $K_i \in \mathcal{T}_h$ denotes the $i^{th}$ element of $\mathcal{T}_h$.  Relative to this basis, for $A\in \mathfrak{g}^0_h\subset L(V_h^0,V_h^0)$ we have
\[
A \Big( \sum_{j=1}^N f_j \psi_j \Big) = \sum_{i=1}^N \Big( \sum_{j=1}^N A_{ij} f_j \Big) \psi_i, \quad \forall f = \sum_{j=1}^N f_j \psi_j \in V_h,
\]
where
\begin{equation} \label{Aij}
A_{ij} = \frac{\langle \psi_i, A\psi_j \rangle}{\langle \psi_i, \psi_i \rangle} = \frac{1}{|K_i|} \langle \psi_i, A\psi_j \rangle.
\end{equation}

In what follows, we will abuse notation by writing $A$ for both the operator $A \in \mathfrak{g}_h^0$ and the matrix $A \in \mathbb{R}^{N \times N}$ with entries \eqref{Aij}.  It is immediate from \eqref{gh} that
\[
\mathring{\mathfrak{g}}^0_h=\Big\{ A \in \mathbb{R}^{N \times N} \;\Big|\; \sum_{j=1}^N A_{ij} = 0, \, \forall i, \text{ and } A^\mathsf{T}\Theta + \Theta A = 0\Big\},\quad \mathfrak{g}^0_h=\Big\{ A \in \mathbb{R}^{N \times N}  \;\Big|\;  \sum_{j=1}^N A_{ij} = 0\Big\}
\]
where $\Theta$ is a diagonal $N\times N$ matrix with diagonal entries $\Theta_{ii} = |K_i|$. These are the Lie algebras used in \cite{PaMuToKaMaDe2010} and \cite{BaGB2019}.

\medskip

The next lemma determines the subspace $S_h^0:= \operatorname{Im}\mathsf{A}$ in the case $r=0$. We write $j \in N(i)$ to indicate that $j \neq i$ and $K_i \cap K_j$ is a shared $(n-1)$-dimensional simplex.

\begin{lemma} \label{lemma:Aij}
If $A=A_u$ for some $u \in H_0(\dv, \Omega)\cap L^p(\Omega)^n$, $p>2$, then, for each $i$,
\begin{equation}\label{A_ij_order0}
\begin{aligned} 
A_{ij} &= -\frac{1}{2|K_i|} \int_{K_i \cap K_j} u \cdot n \, {\rm d}s, \quad j \in N(i),\\
A_{ii} &= \frac{1}{2|K_i|} \int_{K_i} \dv u \, {\rm d}x,
\end{aligned}
\end{equation}
and $A_{ij} = 0$ for all other $j$.
\end{lemma}
\begin{proof}
Let $j \in N(i)$ and consider the expression \eqref{consistentapprox} with $f=\psi_j$ and $g=\psi_i$.  All terms vanish except one, giving
\begin{align*}
\langle A\psi_j, \psi_i \rangle 
&= -\int_{K_i \cap K_j} u \cdot \llbracket \psi_j \rrbracket \{\psi_i \} \, {\rm d}s \\
&= -\frac{1}{2} \int_{K_i \cap K_j} u \cdot n \, {\rm d}s.
\end{align*}
Now consider the case in which $i=j$.  Let $\mathcal{E}^0(K_i)$ denote the set of $(n-1)$-simplices that are on the boundary of $K_i$ but in the interior of $\Omega$.  Since $u \cdot n = 0$ on $\partial\Omega$,
\begin{align*}
\langle A\psi_i, \psi_i \rangle 
&= - \sum_{e \in \mathcal{E}^0(K_i)} \int_e u \cdot \llbracket \psi_i \rrbracket \{\psi_i \} \, {\rm d}s, \\
&= \int_{\partial K_i} u \cdot n \frac{1}{2} \, {\rm d}s \\
&= \frac{1}{2} \int_{K_i} \dv u \, {\rm d}x.
\end{align*}
The expressions in \eqref{A_ij_order0} follow from \eqref{Aij}.
Finally, if $j \neq i$ and $j \notin N(i)$, then all terms in \eqref{consistentapprox} vanish when $f=\psi_j$ and $g=\psi_i$.
\end{proof}

\medskip

\begin{remark}\label{case_r=0}{\rm The expressions in \eqref{A_ij_order0} recover the relations  between Lie algebra elements and vector fields used in \cite{PaMuToKaMaDe2010} and \cite{BaGB2019}. In particular, in the incompressible case, using also \eqref{second_statement} in Proposition \ref{A_u_properties}, we have
\[
\operatorname{Im}\mathsf{A}=\big \{A\in \mathring{\mathfrak{g}}^0_h \mid A_{ij}=0,\;\forall j\in N(i)\big\}
\]
which is the nonholonomic constraint used in \cite{PaMuToKaMaDe2010}. Similarly, in the compressible case, using \eqref{first_statement} in Proposition \ref{A_u_properties}, we have
\[
\operatorname{Im}\mathsf{A}= \big\{A\in \mathfrak{g}^0_h \mid A_{ij}=0,\;\forall j\in N(i),\;\; A^\mathsf{T}\Theta + \Theta A\text{ is diagonal} \big\}
\]
which is the nonholonomic constraint used in \cite{BaGB2019}.  By Proposition~\ref{important_prop}, we have $\operatorname{Im}\mathsf{A} \simeq RT_{2r}(\mathcal{T}_h) = RT_0(\mathcal{T}_h)$ in the compressible case.  This is reflected in~(\ref{A_ij_order0}): every off-diagonal entry of $A \in \operatorname{Im}\mathsf{A}$ corresponds to a degree of freedom $\int_e(u\cdot n){\rm d}s$, $e\in \mathcal{E}^0_h$, for $RT_0(\mathcal{T}_h)$ (and every diagonal entry of $A$ is a linear combination thereof).}
\end{remark}

\section{The Lie algebra-to-vector fields map}\label{sec_Lie_to_vector}

In this section we define a Lie algebra-to-vector fields map that associates to a matrix $A\in L(V_h^r,V_h^r)$ a vector field on $\Omega$. Such a map is needed to define in a general way the semidiscrete Lagrangian associated to a given continuous Lagrangian.

Since any $A\in S_h^r$ is associated to a unique vector field $u \in RT_{2r}(\mathcal{T}_h)$, one could think that the correspondence $A\in S^h_r\rightarrow u\in RT_{2r}(\mathcal{T}_h)$ can be used as a Lie algebra-to-vector fields map. However, as explained in detail in Appendix \ref{Appendix_B}, the Lagrangian must be defined on a larger space than the constraint space $S_h^r$, namely, at least on $S_h^r+[S_h^r,S_h^r]$. This is why such a Lie algebra-to-vector fields map is needed.

\begin{definition} For $r\geq 0$ an integer, we consider the Lie algebra-to-vector field map $\widehat{\;\;}: L(V_h^r, V_h^r) \rightarrow [V_h^r]^n$ defined by
\begin{equation}\label{hat}
\widehat{A}:= \sum_{k=1}^n A (I_h^r(x^k)) e_k,
\end{equation}
where $I_h^r:L^2(\Omega) \rightarrow V_h^r$ is the $L^2$-orthogonal projector onto $V_h^r$, $x^k:\Omega\rightarrow\mathbb{R}$ are the coordinate maps, and $e_k$ the canonical basis for $\mathbb{R}^n$.
\end{definition}

\medskip

The idea leading to the definition \eqref{hat} is the following. On one hand the component $u^k$ of a general vector field $u= \sum_k u^ke_k$, can be understood as the derivative of the coordinate function $x^k$ in the direction $u$, i.e. $u^k= \nabla_u x^k$. On the other hand, from the definition of the discrete diffeomorphism group, the linear map $f\mapsto Af$ for $f\in V_h^r$ is understood as a derivation, hence \eqref{hat} is a natural candidate for a Lie algebra-to-vector field map. We shall study its properties below, after describing in more detail in the next lemma the expression \eqref{hat} for $r=0$.

\medskip

\begin{lemma}\label{hat_0} For $r=0$ and $A\in \mathfrak{g}_h^0\subset L(V_h^0,V_h^0)$, $\widehat{A}$ is the vector field constant on each simplex, given on simplex $K_i$ by
\begin{equation} \label{hatAlow}
\widehat{A} |_{K_i}= \sum_j (b_j - b_i) A_{ij},
\end{equation}
where $b_i = \frac{1}{|K_i|}\int_{K_i} x \, {\rm d}x$
denotes the barycenter of $K_i$.
\end{lemma}
\begin{proof}
The $L^2$-projection of the coordinate function $x^k$ onto $V_h$ is given by
\begin{align*}
I_h^0( x^k) 
&= \sum_j \psi_j \frac{1}{|K_j|} \int_{K_j} x^k = \sum_j \psi_j (b_j)_k,
\end{align*}
where $(b_j)_k$ denotes the $k^{th}$ component of $b_j$.  Hence,
\begin{align*}
\widehat{A} 
&= \sum_{k=1}^n (A(I_h x^k)) e_k 
= \sum_{k=1}^n \sum_j (b_j)_k e_k A \psi_j 
= \sum_j b_j A \psi_j \\
&= \sum_j b_j \sum_i A_{ij} \psi_i 
= \sum_i \psi_i \sum_j b_j A_{ij} 
= \sum_i \psi_i \sum_j (b_j - b_i) A_{ij},
\end{align*}
where the last equality follows from the fact that $\sum_j A_{ij} = 0$ for every $i$.
\end{proof}

\medskip

\begin{proposition}\label{hat_A_u}
For $u\in H_0(\dv,\Omega) \cap L^p(\Omega)$, $p>2$ and $r\geq 0$, we consider $A_u \in L(V_h^r, V_h^r)$ defined in \eqref{consistentapprox}.
\begin{itemize}
\item If $r\geq 1$, then for all $u\in H_0(\dv, \Omega)\cap L^p(\Omega)$, $p>2$, we have 
\[
(\widehat{A_u})^k= I_h^r(u^k),\quad k=1,...,n.
\]
In particular, if $u$ is such that $u|_K \in P_r(K)^n$ for all $K$, then $\widehat{A_u}= u$.
\item If $r=0$, then
\[
\widehat{A_u}|_K = \frac{1}{2|K|}\sum_{e\in K}\int_e u \cdot n_{e_-} (b_{e_+}-b_{e_-}){\rm d}s
\]
where $n_{e_-}$ is the normal vector field pointing from $K_-$ to $K_+$ and $b_{e_\pm}$ are the barycenters of $K_\pm$.
In particular, if $u \in RT_0(\mathcal{T}_h)$, then
\[
\widehat{A_u}|_K = \frac{1}{2|K|}\sum_{e\in K}|e| u \cdot n_{e_-} (b_{e_+}-b_{e_-}).
\]
More particularly,  if $u \in RT_0(\mathcal{T}_h)$ and the triangles are regular
\[
\widehat{A_u}=u.
\]
\end{itemize}
\end{proposition}

As a consequence, we also note that for $r\geq 1$ and $u\in H_0(\dv,\Omega)\cap L^p(\Omega)$, $p>2$ :
\[
\widehat{A_u}=u\;\;\Leftrightarrow\;\; u|_K \in P_r(K)^n,\;\forall K.
\]

\begin{proof} When $r\geq 1$, we have $I_h ^r(x^k)=x^k$, hence $\widehat{A_u} (x)=\sum_{k=1}^n (A_u x^k)(x)e_k$. We compute $A_u x^k\in V_h^r$ as follows: for all $g \in V_h^r$, we have
\[
\langle A_u x^k, g \rangle= \sum_{K\in \mathcal{T}_h}\int_K (\nabla x^k \cdot u )g{\rm d} x- \sum_{e\in \mathcal{E}^0_h}\int_e u \cdot \llbracket x^k\rrbracket_e \{g\}_e{\rm d}s= \int_\Omega u^k g {\rm d}x.
\]
Since this is true for all $g\in V_h^r$ and since $A_u x^k$ must belong to $V_h^r$, we have
\[
(\widehat{A_u})^k = A_u x^k= I_h ^r(u^k),
\]
as desired.

When $r=0$, we have $I_h^0( f)|_{K_i}= \frac{1}{|K_i|}\int_{K_i} f(x) {\rm d }x$ hence $I_h^0(x^k)|_{K_i}= (b_i)^k$.  We compute $A_u I _h^0(x^k)\in V_h^0$ as follows: for all $g \in V_h^0$, we have
\begin{align*}
\langle A_u  I _h^0(x^k), g\rangle &= \sum_{K_i\in\mathcal{T}_h} \int_{K_i}(\nabla (b_i)^k\cdot u ) g {\rm d}x- \sum_{e\in \mathcal{E}_h^0} \int_e u \cdot \left( b_{e_+}^k n_{e_+}+b_{e_-}^k n_{e_-}\right)\{g\}_e{\rm d}s \\
&= 0 - \sum_{e\in \mathcal{E}_h^0} \int_e u \cdot \left( b_{e_+}^k n_{e_+}+b_{e_-}^k n_{e_-}\right)\frac{1}{2} (g_{e_+}+ g_{e_-}) {\rm d}s\\
&= - \sum_{e\in \mathcal{E}_h^0} \int_e u \cdot \left( b_{e_+}^k n_{e_+}+b_{e_-}^k n_{e_-}\right){\rm d}s \left(\frac{1}{2|K_{e_+}|} \int_{K_{e_+}}g{\rm d} x + \frac{1}{2|K_{e_-}|} \int_{K_{e_-}}g{\rm d} x \right)\\
&= -\sum_K \sum_{e\in K} \int_e u \cdot \left( b_{e_+}^k n_{e_+}+b_{e_-}^k n_{e_-}\right){\rm d}s\frac{1}{2|K|} \int_{K}g{\rm d} x
\end{align*}
hence we get
\[
A_u  I _h^0(x^k)|_K= \frac{1}{2|K|} \sum_{e\in K} \int_e u \cdot n_{e_-}\left( b_{e_+}^k -b_{e_-}^k\right){\rm d}s
\]
from which the result follows. This result can be also obtained by combining the results of Proposition \ref{lemma:Aij} and Lemma \ref{hat_0}.

In 2D, for the case of a regular triangle, we have $b_{e_+}-b_{e_-}= n_{e_-} \frac{2}{3}H$, where $H$ is the height, and $|K|= \frac{1}{2}|e| H$ so we get
\[
\widehat{A_u}|_K = \frac{1}{2|K|}\frac{2}{3}H\sum_{e\in K}|e| (u \cdot n_{e_-})n_{e_-}= \frac{2}{3}\sum_{e\in K} (u \cdot n_{e_-})n_{e_-}=u.
\]
Similar computations hold in 3D.
\end{proof}

\medskip

\begin{proposition} \label{prop:hatbracket} For all $u,v\in H_0(\dv, \Omega)\cap L^p(\Omega)$, $p>2$, and $r\geq 1$, we have
\[
\langle \widehat{[A_u, A_v]}^k, g \rangle = \sum_K \int_K ( \nabla \bar v ^k \cdot u - \nabla \bar u^k \cdot v) g {\rm d}x - \sum_{e\in \mathcal{E}^0_h} \int_e\big(u \cdot n [\bar v ^k] - v\cdot n [\bar u ^k] \big)\{g\}{\rm d}s,
\]
for $k=1,...,n$, for all $g\in V_h^r$, where $\bar u ^k= I_h ^r(u^k)\in V_h^r$ and $\bar v^k = I_h^r( v^k)\in V_h^r$. The convention is such that if $n$ is pointing from $K_-$ to $K_+$, then $[\bar v ^k]=\bar v^k_- - \bar v^k_+$.

So, in particular if $u|_K, v|_K\in P_r(K)$, then
\[
\langle \widehat{[A_u, A_v]}^k, g \rangle = \sum_K \int_K [u,v]^k g {\rm d}x - \sum_{e\in \mathcal{E}^0_h} \int_e\big(u \cdot n [ v ^k] - v\cdot n [ u ^k]\big)\{g\}{\rm d}s.
\]

If $u,v,w\in H_0(\dv, \Omega)\cap L^p(\Omega)$, $p>2$ and $u|_K, v|_K, w|_K \in P_r(K)$, we have
\begin{align*}
\int_\Omega\widehat{[A_u, A_v]}\cdot  \widehat{A_w}\,{\rm d}x &= \sum_{k=1}^n \langle \widehat{[A_u, A_v]}^k, \widehat{A_w}^k\rangle\\
& =\sum_K \int_K  [ u,v]\cdot w\, {\rm d}x - \sum_{e\in \mathcal{E}^0_h} \int_e (n \times \{w\}) \cdot [u \times v] {\rm d}s.
\end{align*}

For $r=0$, and $u,v \in H_0(\dv,\Omega)\cap L^p(\Omega)$, $p>2$ such that $u|_K, v|_K\in P_0(K)$, then $\widehat{[A_u,A_v]}\in [V_h^0]^n$ is the vector field constant on each simplex $K$, given on $K$ by
\[
\widehat{[A_u,A_v]}|_K= \frac{1}{2|K|} \sum_{e\in K} |e| \Big(u \cdot n_{e_-}( c[v]_{e_+}- c[v]_{e_-}) - v \cdot n_{e_-}( c[u]_{e_+}- c[u]_{e_-})\Big),
\]
where $c[u]\in [V_h^0]^n$ is the vector field constant on each simplex $K$, given on $K$ by
\[
c[u]_K = \frac{1}{2|K|}\sum_{e\in K}|e| u \cdot n_{e_-} (b_{e_+}-b_{e_-})
\]
similarly for $c[v]\in [V_h^0]^n$.
\end{proposition}
\begin{proof} We note that from Proposition \ref{hat_A_u},
\[
\widehat{[A_u, A_v]}^k= A_u(A_v x^k)- A_v(A_u x^k) = A_u I_h^r (v^k)- A_v I_h ^r(u^k)=A_u \bar v^k- A_v \bar u^k.
\]
for all $u,v\in H_0(\dv,\Omega)\cap L^p(\Omega)$, $p>2$.
Then, using \eqref{consistentapprox}, we have for all $g\in V_h^r$:
\[
\langle A_u \bar v^k, g\rangle =\sum_K \int_K (\nabla \bar v^k \cdot u) g {\rm d}x- \sum_{e\in\mathcal{E}^0_h}\int_e u \cdot \llbracket \bar v^k\rrbracket\{g\}{\rm d}s
\]
similarly for $\langle A_v \bar u^k, g\rangle$ from which we get the first formula.

\medskip

The second formula follows when $u|_K, v|_K \in P_r(K)$ since in this case $u=\bar u$, $v=\bar v$.

For the third formula we choose $g= \widehat{A_w}^k=\bar w^k$ in the first formula and sum over $k=1,...,n$ to get
\begin{align*}
\int_\Omega\widehat{[A_u, A_v]}\cdot  \widehat{A_w}\,{\rm d}x &= \sum_{k=1}^n \langle \widehat{[A_u, A_v]}^k, \widehat{A_w}^k\rangle\\
& =\sum_K \int_K (\nabla \bar v\cdot u - \nabla \bar u \cdot v)\cdot \bar w\, {\rm d}x\\
&\qquad \qquad - \sum_{e\in \mathcal{E}^0_h} \int_e \big(u\cdot n\, [\bar v] - v\cdot n\, [\bar u] \big)\cdot \{\bar w\}  {\rm d}s.
\end{align*}
So far we only used $u,v,w\in H_0(\dv, \Omega)\cap L^p(\Omega)$, $p>2$. Now we assume further that $u|_K, v|_K, w|_K\in P_r(K)$ for all $K$, so we have $\bar u=u$, $\bar v=v$, $\bar w=w$, $u\cdot n=\{u\}\cdot n$, $v\cdot n= \{v\}\cdot n$, $[\bar v]\cdot n= [\bar u]\cdot n=0$. Using some vector calculus identities for the last term, we get
\[
\int_\Omega\widehat{[A_u, A_v]}\cdot  \widehat{A_w}\,{\rm d}x  =\sum_K \int_K [u,v]\cdot w\, {\rm d}x - \sum_{e\in \mathcal{E}^0_h} \int_e  (n \times \{ w\}) \cdot  ( \{u\}\times [ v] + [ u] \times \{v\})  {\rm d}s.
\]
which yields the desired formula since $[u\times v]= \{u\}\times [v]+ [u]\times \{v\}$.
\end{proof}

\section{Finite element variational integrator}\label{sec_FEVA}

In this section we derive the variational discretization for compressible fluids by using the setting developed so far. We focus on the case in which the Lagrangian depends only on the velocity and the mass density, since the extension to a dependence on the entropy density is straightforward, see \S\ref{sec_examples} and Appendix \ref{Appendix_A}.

\subsection{Semidiscrete Euler-Poincar\'e equations}

Given a continuous Lagrangian $\ell(u,\rho)$ expressed in terms of the Eulerian velocity $u$ and mass density $\rho$, the associated discrete Lagrangian $\ell_d:\mathfrak{g}_h^r\times V_h^r\rightarrow \mathbb{R}$ is defined with the help of the Lie algebra-to-vector fields map as 
\begin{equation}\label{discrete_l}
\ell_d(A, D):= \ell(\widehat{A}, D),
\end{equation}
where $D\in V_h^r$ is the discrete density.
Exactly as in the continuous case, the right action of $G_h^r$ on discrete densities is defined by duality as
\begin{equation} \label{Daction}
\langle D \cdot q, E \rangle  = \langle D, qE \rangle, \quad \forall E \in V_h^r.
\end{equation}
The corresponding action of $\mathfrak{g}_h^r$ on $D$ is given by
\begin{equation} \label{Daction_g}
\langle D \cdot B, E \rangle  = \langle D, BE \rangle, \quad \forall E \in V_h^r.
\end{equation}

The semidiscrete equations are derived by mimicking the variational formulation of the continuous equations, namely, by using the Euler-Poincar\'e principle applied to $\ell_d$. As we have explained earlier, only the Lie algebra elements in $\operatorname{Im}\mathsf{A}=S_h^r$ actually represent a discretization of continuous vector fields. Following the approach initiated \cite{PaMuToKaMaDe2010} this condition is included in the Euler-Poincar\'e principle by imposing $S_h^r$ as a nonholonomic constraint, and hence applying the Euler-Poincar\'e-d'Alembert recalled in Appendix \ref{Appendix_B}. As we will see later, one needs to further restrict the constraint $S_h^r$ to a subspace $\Delta_h^R\subset S_h^r$.

\medskip

For a given constraint $\Delta_h^R\subset \mathfrak{g}_h^r$, a given Lagrangian $\ell_d$, and a given duality pairing $\langle\!\langle K,A \rangle\!\rangle$ between elements $K\in (\mathfrak{g}_h^r)^*$ and $A\in \mathfrak{g}_h^r$, the Euler-Poincar\'e-d'Alembert principle seeks $A(t)\in \Delta _h^R$ and $D(t)\in V_h^r$ such that
\[
\delta \int_0^T\ell_d(A, D){\rm d}t=0,\quad \text{for $\delta A = \partial_t B + [B,A]$ and $\delta D= - D\cdot B$},
\]
for all $B(t) \in \Delta_h^R$ with $B(0)=B(T)=0$.  The expressions for $\delta A$ and $\delta B$ are deduced from the relations $A(t)= \dot q(t) q(t)^{-1}$ and $D(t)= D_0\cdot q(t)^{-1}$, with $D_0$ the initial value of the density, as in the continuous case.

The critical condition associated to this principle is
\begin{equation}\label{EP_a_weak_fluid}
\Big\langle\!\!\Big\langle \partial_t \frac{\delta\ell_d}{\delta A}, B\Big\rangle\!\!\Big\rangle + \Big\langle\!\!\Big\langle \frac{\delta\ell_d}{\delta A}, [A,B]\Big\rangle\!\!\Big\rangle + \Big\langle \frac{\delta\ell_d}{\delta D},D\cdot B \Big\rangle=0, \quad \forall \;t \in (0,T),\quad\forall\; B\in \Delta_h^R,
\end{equation}
or, equivalently,
\begin{equation}\label{EP_a_strong_fluid}
 \partial_t \frac{\delta\ell_d}{\delta A}+ \operatorname{ad}^*_A \frac{\delta\ell_d}{\delta A} - \frac{\delta\ell_d}{\delta D}\diamond D\in (\Delta_h^R)^\circ , \quad \forall t \in (0,T).
\end{equation}
The differential equation for $D$ follows from differentiating $D(t)=D_0\cdot q(t)^{-1}$ to obtain $\partial_t D = -D \cdot A$, or, equivalently,
\begin{equation} \label{Devolution}
\langle \partial_t D, E \rangle + \langle D, AE \rangle = 0, \quad \forall t \in (0,T),\quad\forall\; E\in V_h^r.
\end{equation}
We refer to Appendix \ref{Appendix_B} for more details and the explanation of the notations. The extension of \eqref{EP_a_weak_fluid} and \eqref{EP_a_strong_fluid} to the case when the Lagrangian depends also on the entropy density is straightforward but important, see \S\ref{sec_examples}.

\medskip

As explained in Appendix \ref{Appendix_B}, a sufficient condition for \eqref{EP_a_weak_fluid} to be a solvable system for $T$ small enough is that the map 
\begin{equation}\label{diffeomorphism}
A\in \Delta^R_h \rightarrow \frac{\delta \ell_d}{\delta A}(A,D)\in (\mathfrak{g}_h^r)^*/(\Delta^R_h)^\circ
\end{equation}
is a diffeomorphism for all $D\in V^r_h$ strictly positive.

\subsection{The compressible fluid}\label{sec_comp}

We now focus on the compressible barotropic fluid, whose continuous Lagrangian is given in \eqref{Lagra_comp}.  Following \eqref{discrete_l}, we have the discrete Lagrangian
\begin{equation}\label{discrete_l_fluid}
\ell_d(A, D):= \ell(\widehat{A}, D)= \int_\Omega \Big[\frac{1}{2}D |\widehat{A}|^2- D e( D)\Big]{\rm d} x.
\end{equation}
In order to check condition \eqref{diffeomorphism}, we shall compute the functional derivative $\frac{\delta \ell_d}{\delta A}$. We have
\[
\Big\langle\!\!\Big\langle \frac{\delta\ell_d}{\delta A},\delta A \Big\rangle\!\!\Big\rangle = \int_\Omega D \widehat{A}\cdot \widehat{\delta A} {\rm d}x = \int_\Omega I_h^r(D \widehat{A})\cdot \widehat{\delta A} {\rm d}x =\Big\langle\!\!\Big\langle  I_h^r(D\widehat{A})^\flat,\delta A\Big\rangle\!\!\Big\rangle
\]
where we defined the linear map $\flat: ([V_h^r]^n)^*=[V_h^r]^n\rightarrow (\mathfrak{g}_h^r)^*$ as the dual map to $\widehat{\;}:\mathfrak{g}_h^r\rightarrow [V_h^r]^n$, namely
\[
\langle\!\langle \alpha^\flat , A \rangle\!\rangle =  \langle\alpha,\widehat{A} \rangle,\;\forall \alpha\in [V_h^r]^n, \; A\in \mathfrak{g}_h^r.
\]
We thus get $\frac{\delta \ell_d}{\delta A}=   I_h^r(D\widehat{A})^\flat$ and note that the choice $\Delta_h^R=S_h^r$ is not appropriate since the linear map $A\in S^r_h \mapsto    I_h^r(D\widehat{A})^\flat\in (\mathfrak{g}_h^r)^*/(S^r_h)^\circ$ is not an isomorphism.
We thus need to restrict the constraint $S_h^r$ to a subspace $\Delta^R_h\subset S_h^r$ such that
\begin{equation}\label{isom}
A\in \Delta^R_h \mapsto    I_h^r(D\widehat{A})^\flat\in (\mathfrak{g}_h^r)^*/(\Delta^R_h)^\circ
\end{equation}
becomes an isomorphism, for all $D\in V^r_h$ strictly positive. We shall denote by $R_h$ the subspace of $RT_{2r}(\mathcal{T}_h)$ corresponding to $\Delta^R_h$ via the isomorphism $u\in RT_{2r}(\mathcal{T}_h)\mapsto A_u \in S_h^r$ shown in Proposition \ref{important_prop}.
The diagram below illustrates the situation that we consider.
\[
\begin{xy}
\xymatrix{
& & H_0(\dv,\Omega) \textcolor{white}{\frac{1}{2}}\ar[r]^{\;\;\mathsf{A}} & S_h^r \textcolor{white}{\frac{1}{2}}\ar@{^{(}->}[r] & \mathfrak{g}_h^r \ar[r]^{\widehat{\;}} & [V_h^r]^n\\
& & RT_{2r}(\mathcal{T}_h)\textcolor{white}{\frac{1}{2}}\ar@{^{(}->}[u]  \ar@{<->}[ur]& \Delta_h^R \textcolor{white}{\frac{1}{2}}\ar@{^{(}->}[u]& &\\
& & R_h\textcolor{white}{\frac{1}{2}}\ar@{^{(}->}[u] \ar@{<->}[ur]& & &
}
\end{xy}
\]

The kernel of \eqref{isom} is computed as
\begin{align*}
\{A \in \Delta^R_h \mid I_h^r(D\widehat{A})^\flat\in (\Delta^R_h)^\circ \}&= \{A\in \Delta^R_h  \mid \langle\!\langle I_h^r(D\widehat{A})^\flat, B \rangle\!\rangle=0,\;\forall B\in \Delta^R_h \}\\
&= \{A\in \Delta^R_h  \mid \langle I_h^r(D\widehat{A}), \widehat{B} \rangle=0,\;\forall B\in \Delta^R_h \}\\
&=\{A_u \in  \Delta^R_h \mid \langle I_h^r(DI_h^r(u)), I^r_h(v) \rangle=0,\;\forall v\in R_h \} \\
&=\{A_u \in  \Delta^R_h \mid \langle DI_h^r(u), I^r_h(v) \rangle=0,\;\forall v\in R_h \}.
\end{align*}
We note that since $A,B\in \Delta^R_h \subset S_h^r$, we have $A=A_u$ and $B=B_v$ for unique $u,v \in R_h\subset RT_{2r}(\mathcal{T}_h)$ by Proposition \ref{important_prop}, so the kernel is isomorphic to the space
\begin{equation}\label{space_intermediate}
\{u \in R_h  \mid \langle DI_h^r(u), I^r_h(v) \rangle=0,\;\forall v\in R_h \}.
\end{equation}
This space is zero if and only if $R_h$ is a subspace of $[V_h^r]^n\cap H_0(\dv, \Omega)=BDM_r(\mathcal{T}_h)$, the Brezzi-Douglas-Marini finite element space of order $r$. Indeed, in this case the space \eqref{space_intermediate} can be rewritten as
\[
\{u \in R_h  \mid \langle Du, v \rangle=0,\;\forall v\in R_h \}=\{0\}
\]
since $D$ is strictly positive (it suffices to take $v=u$).  Conversely, if there exists a nonzero $w \in R_h \setminus BDM_r(\mathcal{T}_h)$, then $u := w - I_h^r(w) \neq 0$ satisfies $I_h^r(u) = 0$, showing that \eqref{space_intermediate} is nonzero.

\medskip

Using the expressions of the functional derivatives
\[
\frac{\delta \ell_d}{\delta A}=   I_h^r(D\widehat{A})^\flat,\quad \frac{\delta \ell_d}{\delta D}=I_h^r\Big( \frac{1}{2} |\widehat{A}|^2-  e( D) -D \frac{\partial e}{\partial D}\Big),
\]
of  \eqref{discrete_l_fluid}, the Euler-Poincar\'e equations \eqref{EP_a_weak_fluid} are equivalent to
\begin{equation}\label{EP_comp}
\Big\langle \partial_t (D\widehat{A}), \widehat{B}\Big\rangle  + \Big\langle  D\widehat{A}, \widehat{[A,B]} \Big\rangle + \Big\langle I_h^r\Big(\frac{1}{2} |\widehat{A}|^2-  e( D) -D \frac{\partial e}{\partial D}\Big),D\cdot B \Big\rangle=0, \quad \forall t \in (0,T), \quad\forall\; B\in \Delta_h^R.
\end{equation}

To relate \eqref{EP_comp} and \eqref{Devolution} to more traditional finite element notation, let us denote $\rho_h = D$, $u_h = -\widehat{A}$, and $\sigma_h = E$, and $v_h=-\widehat{B}$.  Then, using Proposition~\ref{prop:hatbracket}, the identities $\widehat{A_{u_h}} = -\widehat{A}$ and $\widehat{A_{v_h}} = -\widehat{B}$, and the definition  \eqref{consistentapprox} of $A_u$, we see that  \eqref{EP_comp} and \eqref{Devolution} are equivalent to seeking $u_h \in R_h$ and $\rho_h \in V_h^r$ such that
\begin{align}
\langle \partial_t(\rho_h u_h), v_h \rangle + a_h(w_h, u_h, v_h) - b_h(v_h, f_h, \rho_h) &= 0, & \forall v_h \in R_h \label{rhohuhdot} \\
\langle \partial_t \rho_h, \sigma_h \rangle - b_h(u_h, \sigma_h, \rho_h) &= 0, & \forall \sigma_h \in V_h^r, \label{rhohdot}
\end{align}
where $w_h = I_h^r(\rho_h u_h)$, $f_h = I_h^r\left( \frac{1}{2}|u_h|^2 - e(\rho_h) - \rho_h \frac{\partial e}{\partial \rho_h} \right)$, and
\begin{align*}
a_h(w,u,v) &= \sum_{K \in \mathcal{T}_h} \int_K w \cdot (v \cdot \nabla u - u \cdot \nabla v) \, dx + \sum_{e \in \mathcal{E}_h^0} \int_e (v \cdot n [ u ]- u \cdot n [ v ]) \cdot \{w\} \, ds, \\
b_h(w,f,g) &= \sum_{K \in \mathcal{T}_h} \int_K (w \cdot \nabla f) g \, dx - \sum_{e \in \mathcal{E}_h^0} w \cdot \llbracket f \rrbracket \{g\} \, ds.
\end{align*}

\begin{remark}
The above calculations carry over also to the setting in which the density is taken to be an element of $V_h^s \subset V_h^r$, $s<r$.  In this setting,~(\ref{rhohdot}) must hold for every $\sigma_h \in V_h^s$, the definition of $f_h$ becomes $f_h = I_h^s\left( \frac{1}{2}|u_h|^2 - e(\rho_h) - \rho_h \frac{\partial e}{\partial \rho_h} \right)$, and the definition of $w_h$ remains unchanged.  By fixing $s$ and $R_h$, we may then take $r$ large enough so that $I_h^r(\rho_h u_h)=\rho_h u_h$.
\end{remark}

\paragraph{Extension to rotating fluids.} For the purpose of application in geophysical fluid dynamics, we consider the case of a rotating fluid with angular velocity $\omega$ in a gravitational field with potential $\Phi(x)$. The equations of motion are obtained by taking the Lagrangian
\[
\ell(u,\rho)= \int_\Omega \Big[\frac{1}{2}\rho|u|^2 + \rho R\cdot u  - \rho e(\rho) -\rho\Phi\Big] {\rm d}x,
\]
where the vector field $R$ is half the vector potential of $\omega$, i.e. $2\omega= \operatorname{curl}R$. Application of the Euler-Poincar\'e principle \eqref{variationalu_rho} yields the balance of fluid momentum
\begin{equation}\label{Euler_rot}
\rho(\partial_t u + u \cdot \nabla u + 2\omega \times u) = - \rho\nabla\Phi -\nabla p, \quad \text{with}\quad p = \rho^2\frac{\partial e}{\partial \rho}.
\end{equation}

The discrete Lagrangian is defined exactly as in \eqref{discrete_l} and reads
\begin{equation}\label{discrete_l_rot}
\ell_d(A, D):= \ell(\widehat{A}, D)= \int_\Omega \Big[\frac{1}{2}D |\widehat{A}|^2+ D \widehat{A}\cdot R- D e( D) - D \Phi \Big]{\rm d} x.
\end{equation}
We get $\frac{\delta \ell_d}{\delta A}=   I_h^r(D\widehat{A})^\flat+ I_h (D R)^\flat$ and the same reasoning as before shows that the affine map
\begin{equation}\label{isom_rot}
A\in \Delta^R_h \rightarrow   \frac{\delta \ell_d}{\delta A}\in (\mathfrak{g}_h^r)^*/(\Delta^R_h)^\circ
\end{equation}
is a diffeomorphism for all $D\in V_h^r$ strictly positive. The Euler-Poincar\'e equations \eqref{EP_a_weak_fluid} now yield
\begin{equation}\label{EP_comp_rot}
\Big\langle \partial_t \big(D(\widehat{A}+R)\big), \widehat{B}\Big\rangle  + \Big\langle  D(\widehat{A}+R), \widehat{[A,B]} \Big\rangle + \Big\langle I_h^r\Big(\frac{1}{2} |\widehat{A}|^2 + \widehat{A} \cdot R-  e( D) -D \frac{\partial e}{\partial D}-\Phi\Big),D\cdot B \Big\rangle=0,\;\;\forall B\in \Delta^R_h,
\end{equation}
which, in traditional finite element notations is
\begin{align}
\langle \partial_t(\rho_h u_h+\rho_h R), v_h \rangle + a_h(w_h, u_h, v_h) - b_h(v_h, f_h, \rho_h) &= 0, &\quad \forall v_h \in R_h, \label{FEMvelocity} \\
\langle \partial_t \rho_h, \sigma_h \rangle - b_h(u_h, \sigma_h, \rho_h) &= 0, &\quad \forall \sigma_h \in V_h^r, \label{FEMdensity}
\end{align}
where $w_h = I_h^r(\rho_h u_h+\rho_h R)$, $f_h = I_h^r\left( \frac{1}{2}|u_h|^2 + u_h \cdot R - e(\rho_h)- \rho_h \frac{\partial e}{\partial \rho_h} -\Phi \right)$, and $a_h$, $b_h$ defined as before.

\paragraph{Lowest-order setting.} As a consequence of Remark \ref{case_r=0}, for $r=0$, the Euler-Poincar\'e equations \eqref{EP_a_weak_fluid} are identical to the discrete equations considered in \cite{BaGB2019} and, in the incompressible case, they coincide with those of \cite{PaMuToKaMaDe2010,GaMuPaMaDe2011,DeGaGBZe2014}. The discrete Lagrangians used are however different. For instance, by using the result of Lemma \ref{hat_0}, for $r=0$, the discrete Lagrangian \eqref{discrete_l_fluid} for $A\in S_h^0$ becomes
 \begin{equation} \label{lagrangianAlow}
\ell(A,D) = \frac{1}{2} \sum_i |K_i| D_i\sum_{j,k \in N(i)} M^{(i)}_{jk} A_{ij}A_{ik}- \sum_i |K_i| D_i e(D_i),
\end{equation}
where $M^{(i)}_{jk} = (b_j - b_i) \cdot (b_k - b_i)$.
This is similar, but not identical, to the reduced Lagrangian used in, e.g., \cite{PaMuToKaMaDe2010}.  There, each $M^{(i)}$ is replaced by a diagonal matrix with diagonal entries $M^{(i)}_{jj} = 2|K_i||c_i-c_j|/|K_i \cap K_j|$, where $c_i$ denotes the circumcenter of $K_i$.

\paragraph{Variational time discretization.} The variational character of compressible fluid equations can be exploited also at the temporal level, by deriving the temporal scheme via a discretization in time of the Euler-Poincar\'e variational principle, in a similar way to what has been done in \cite{GaMuPaMaDe2011,DeGaGBZe2014} for incompressible fluid models. This discretization of the Euler-Poincar\'e equation follows the one presented in \cite{BRMa2009}.

In this setting, the relations $A(t)=\dot g (t) g(t)^{-1}$ and $D(t)=D_0\cdot g(t)^{-1}$ are discretized as
\begin{equation}\label{discrete_reduction}
A_k=\tau^{-1}(g_{k+1}g_k^{-1})/\Delta t\quad\text{and}\quad D_k=D_0\cdot g_k^{-1},
\end{equation}
where $\tau:\mathfrak{g}_h^r\rightarrow G_h^r$ is a local diffeomorphism from a neighborhood of $0 \in \mathfrak{g}_h^r$ to a neighborhood of $e \in G_h^r$ with $\tau(0)=e$ and $\tau(A)^{-1}=\tau(-A)$. Given $A\in\mathfrak{g}_h^r$, we denote by $d\tau_A:\mathfrak{g}_h^r\rightarrow\mathfrak{g}_h^r$ the right trivialized tangent map defined as
\[
d\tau_A(\delta A):=\left(\mathbf{D}\tau(A)\cdot\delta A\right)\tau(A)^{-1},\quad\delta A\in\mathfrak{g}_h^r.
\]
We denote by $d\tau_A^{-1}:\mathfrak{g}_h^r\rightarrow\mathfrak{g}_h^r$ its inverse and by $ (d\tau_A^{-1} )^*:( \mathfrak{g}_h^r)^*\rightarrow(\mathfrak{g}_h^r)^*$ the dual map.

The discrete Euler-Poincar\'e-d'Alembert variational principle reads
\[
\delta \sum_{k=0}^{K-1}\ell_d(A_k,D_k)\Delta t=0,
\]
for variations
\[
\delta A_k = \frac{1}{\Delta t}d\tau_{\Delta t A_k}^{-1}(B_{k+1})- \frac{1}{\Delta t}d\tau^{-1}_{- \Delta tA_k}(B_k),\qquad \delta D_k= - D_k \cdot B_k
\]
where $B_k\in \Delta^R_h$ vanishes at the extremities. These variations are obtained by taking the variations of the relations \eqref{discrete_reduction} and defining $B_k= \delta g_k g_k^{-1}$. It yields
\[
\frac{1}{\Delta t}\Big\langle\!\!\Big\langle\left(d\tau^{-1}_{\Delta t A_{k-1}}\right)^*\frac{\delta\ell_d}{\delta A_{k-1}}-  \left(d\tau^{-1}_{-\Delta t A_k}\right)^*\frac{\delta\ell_d}{\delta A_k}, B_k \Big\rangle\!\!\Big\rangle - \Big\langle \frac{\delta \ell_d}{\delta D_k}, D_k\cdot B_k\Big\rangle=0,\qquad \forall\; B_k \in \Delta_h^R.
\]
From $D_k= D_0\cdot g_k^{-1}$, one gets
\begin{equation} \label{Dkp1}
D_{k+1}=D_k\cdot \tau(-\Delta tA_k).
\end{equation}

Several choices are possible for the local diffeomorphism $\tau$, see, e.g., \cite{BRMa2009}.  One option is the Cayley transform
\[
\tau (A)= \left(I- \frac{A}{2}\right)^{-1} \left(I +\frac{A}{2}\right).
\]
We have $\tau(0)=I$ and since $\mathbf{D}\tau(0)\delta A= \delta A$, it is a local diffeomorphism. We also note that $A\mathbf{1}=0$ implies $\tau(A) \mathbf{1}=\mathbf{1}$ in a suitably small neighborhood of $0$. We have
\[
d\tau_A(\delta A)=  \left(I- \frac{A}{2}\right)^{-1} \delta A  \left(I + \frac{A}{2}\right)^{-1},\quad d\tau_A^{-1}(B)= B + \frac{1}{2}[B,A] - \frac{1}{4}ABA,
\]
so the discrete Euler-Poincar\'e equations read
\begin{align*}
&\Big\langle\!\!\Big\langle \frac{1}{\Delta t}\left(\frac{\delta\ell_d}{\delta A_k}- \frac{\delta\ell_d}{\delta A_{k-1}}\right) , B_k \Big\rangle\!\!\Big\rangle +\frac{1}{2} \Big\langle\!\!\Big\langle \frac{\delta\ell_d}{\delta A_k}, [A_k,B_k] - \frac{\Delta t}{2}  A_k B_k A_k\Big\rangle\!\!\Big\rangle \\
&+ \frac{1}{2} \Big\langle\!\!\Big\langle\frac{\delta\ell_d}{\delta A_{k-1}}, [A_{k-1},B_k] + \frac{\Delta t}{2} A_{k-1} B_k A_{k-1} \Big\rangle\!\!\Big\rangle  + \Big\langle \frac{\delta \ell_d}{\delta D_k}, D_k\cdot B_k\Big\rangle=0,\quad \forall\; B_k \in \Delta_h^R.
\end{align*}
This is the discrete time version of \eqref{EP_a_weak_fluid}. The discrete time version of \eqref{EP_a_strong_fluid} can be similarly written.
With this choice of $\tau$, the evolution $D_k$ is obtained from~(\ref{Dkp1}), which is equivalent to
\[
D_k \cdot (I+\frac{\Delta t}{2} A_{k-1}) = D_{k-1} \cdot (I-\frac{\Delta t}{2} A_{k-1}).
\]
Recalling~(\ref{Daction_g}), we get
\[
\left\langle \frac{D_k-D_{k-1}}{\Delta t}, E_k \right\rangle + \left\langle \frac{D_{k-1}+D_k}{2}, A_{k-1} E_k \right\rangle = 0, \quad \forall E_k \in V_h^r.
\]

\paragraph{Energy preserving time discretization.}
For Lagrangians of the form~(\ref{discrete_l_rot}), it is possible to construct a time discretization that exactly preserves the energy $\int_\Omega [\frac{1}{2}D |\widehat{A}|^2 + D e( D) + D\Phi ]{\rm d} x$.  Note that the contribution of the rotation does not appear in the expression of the total energy. To do this, let us define
\begin{equation}\label{discrete_gradient}
F_{k-1/2} = \frac{1}{2} \widehat{A_{k-1}} \cdot \widehat{A_k} + \widehat{A_{k-1/2}} \cdot R - f(D_{k-1},D_k)- \Phi,
\end{equation}
where
\[
f(x,y) = \frac{ye(y)-xe(x)}{y-x}.
\]
Also let $A_{k-1/2} = \frac{1}{2}(A_{k-1}+A_k)$ and $D_{k-1/2} = \frac{1}{2}(D_{k-1}+D_k)$.  The energy-preserving scheme reads
\begin{align}
&&\Big\langle\!\!\Big\langle \frac{1}{\Delta t}\left(\frac{\delta\ell_d}{\delta A_k}- \frac{\delta\ell_d}{\delta A_{k-1}}\right) , B_k \Big\rangle\!\!\Big\rangle\hspace{2in}& \nonumber \\
&&+\frac{1}{2} \Big\langle\!\!\Big\langle \frac{\delta\ell_d}{\delta A_{k-1}} + \frac{\delta\ell_d}{\delta A_k}, [A_{k-1/2},B_k] \Big\rangle\!\!\Big\rangle  + \langle F_{k-1/2}, D_{k-1/2} \cdot B_k\rangle&=0,\quad \forall\; B_k \in \Delta_h^R, \label{energy_pres_A} \\
&&\left\langle \frac{D_k - D_{k-1}}{\Delta t}, E_k \right\rangle + \langle D_{k-1/2} \cdot A_{k-1/2}, E_k \rangle &= 0, \quad \forall E_k \in V_h^r. \label{energy_pres_D}
\end{align}

\begin{proposition}
The solution of~(\ref{energy_pres_A}-\ref{energy_pres_D}) satisfies
\begin{equation} \label{discrete_energy_pres}
\int_\Omega \left[\frac{1}{2}D_k |\widehat{A_k}|^2 + D_k e( D_k) + D_k\Phi \right]{\rm d} x = \int_\Omega \left[\frac{1}{2}D_{k-1} |\widehat{A_{k-1}}|^2 + D_{k-1} e( D_{k-1}) + D_{k-1}\Phi \right]{\rm d} x.
\end{equation}
\end{proposition}
\begin{proof}
Taking $B_k = A_{k-1/2}$ in~(\ref{energy_pres_A}) gives
\begin{align*}
\Big\langle\!\!\Big\langle \frac{1}{\Delta t}\left(\frac{\delta\ell_d}{\delta A_k}- \frac{\delta\ell_d}{\delta A_{k-1}}\right) , A_{k-1/2} \Big\rangle\!\!\Big\rangle + \langle F_{k-1/2}, D_{k-1/2} \cdot A_{k-1/2} \rangle = 0.
\end{align*}
Using the density equation~(\ref{energy_pres_D}) and the definition~(\ref{discrete_l_rot}) of $\ell_d$, we can rewrite this as
\begin{align}
\Big\langle \frac{1}{\Delta t}\left(D_k(\widehat{A_k}+R) - D_{k-1}(\widehat{A_{k-1}}+R)\right) , \widehat{A_{k-1/2}} \Big\rangle - \left\langle \frac{D_k-D_{k-1}}{\Delta t}, F_{k-1/2} \right\rangle = 0. \label{energyproof}
\end{align}
After rearrangement, the first term can be expressed as
\begin{align*}
&\Big\langle \frac{1}{\Delta t}\left(D_k(\widehat{A_k}+R) - D_{k-1}(\widehat{A_{k-1}}+R)\right) , \widehat{A_{k-1/2}} \Big\rangle \\
&=
\frac{1}{2\Delta t} \left( \langle  D_k \widehat{A_k}, \widehat{A_k} \rangle - \langle  D_{k-1} \widehat{A_{k-1}}, \widehat{A_{k-1}} \rangle \right) + \left\langle \frac{D_k-D_{k-1}}{\Delta t}, \widehat{A_{k-1/2}} \cdot R + \frac{1}{2}\widehat{A_{k-1}} \cdot \widehat{A_k}  \right\rangle.
\end{align*}
Inserting this and the definition of $F_{k-1/2}$ into~(\ref{energyproof}) gives
\[
\frac{1}{2\Delta t} \left( \langle  D_k \widehat{A_k}, \widehat{A_k} \rangle - \langle  D_{k-1} \widehat{A_{k-1}}, \widehat{A_{k-1}} \rangle \right) + \left\langle \frac{D_k-D_{k-1}}{\Delta t}, f(D_{k-1},D_k) + \Phi \right\rangle = 0.
\]
Finally, the definition of $f$ yields
\[
\frac{1}{2\Delta t} \left( \langle  D_k \widehat{A_k}, \widehat{A_k} \rangle - \langle  D_{k-1} \widehat{A_{k-1}}, \widehat{A_{k-1}} \rangle \right) + \left\langle \frac{D_k e(D_k)-D_{k-1}e(D_{k-1}) + D_k \Phi -D_{k-1}\Phi}{\Delta t}, 1 \right\rangle = 0,
\]
which is equivalent to~(\ref{discrete_energy_pres}).
\end{proof}

Note that the definition of $F_{k-1/2}$ in \eqref{discrete_gradient} can be rewritten in terms of $\ell_d$ as
\begin{equation} \label{elldiff}
\ell_d(A_k, D_k) - \ell_d(A_{k-1},D_{k-1})= \frac{1}{2} \Big\langle\!\!\Big\langle \frac{\delta\ell_d}{\delta A_{k-1}} + \frac{\delta\ell_d}{\delta A_k}, A_k-A_{k-1} \Big\rangle\!\!\Big\rangle+ \langle F_{k-1/2}, D_k - D_{k-1} \rangle,
\end{equation}
This is reminiscent of a discrete gradient method~\cite[p. 174]{HaLuWa2006}, with $F_{k-1/2}$ playing the role of the discrete version of $\frac{\delta\ell}{\delta D}$.

\section{Numerical tests}\label{sec_examples}

\begin{table}[t]
\centering
\maketable{rotating_a0b0.dat}
\caption{$L^2$-errors in the velocity and density at time $T=0.5$.}
\label{tab:a0b0}
\end{table}

\begin{table}[t]
\centering
\maketabledt{rotating_a0b0dt.dat}
\caption{Convergence with respect to $\Delta t$ of the $L^2$-errors in the velocity and density at time $T=0.5$.}
\label{tab:a0b0_dt}
\end{table}

\begin{figure}
\centering
\includegraphics[scale=0.25,trim=9in 2in 9in 2in,clip=true]{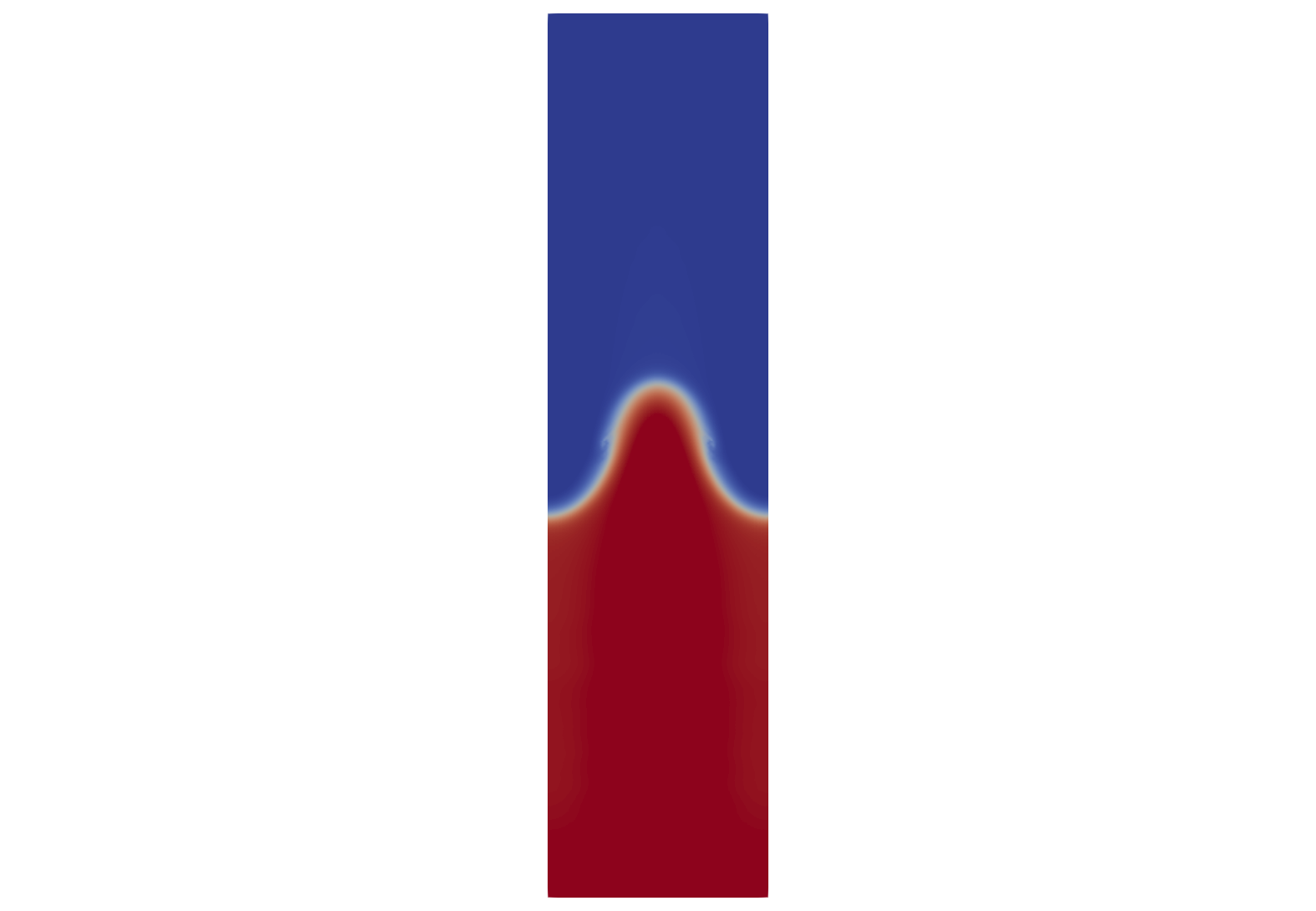}
\includegraphics[scale=0.25,trim=9in 2in 9in 2in,clip=true]{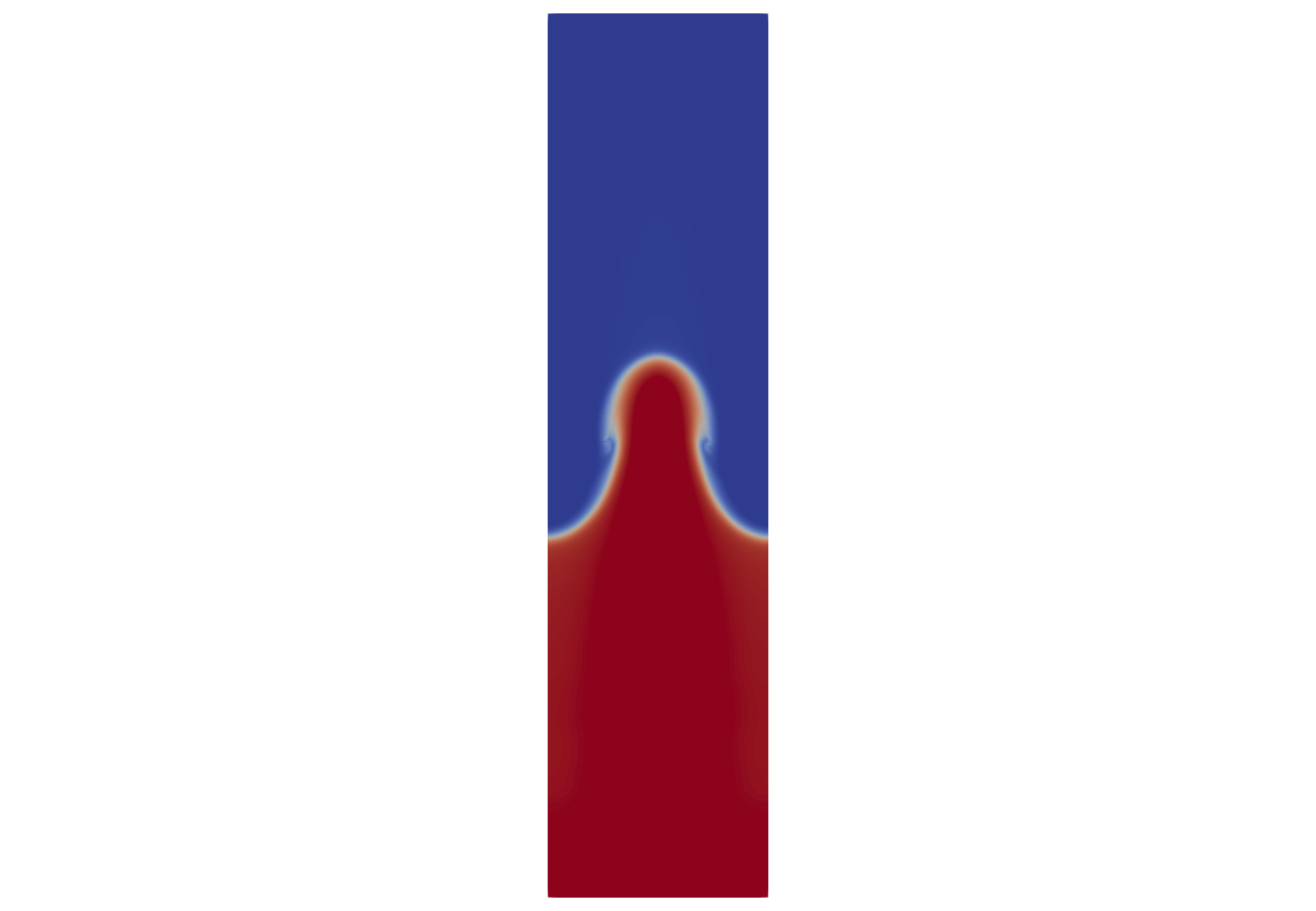}
\includegraphics[scale=0.25,trim=9in 2in 9in 2in,clip=true]{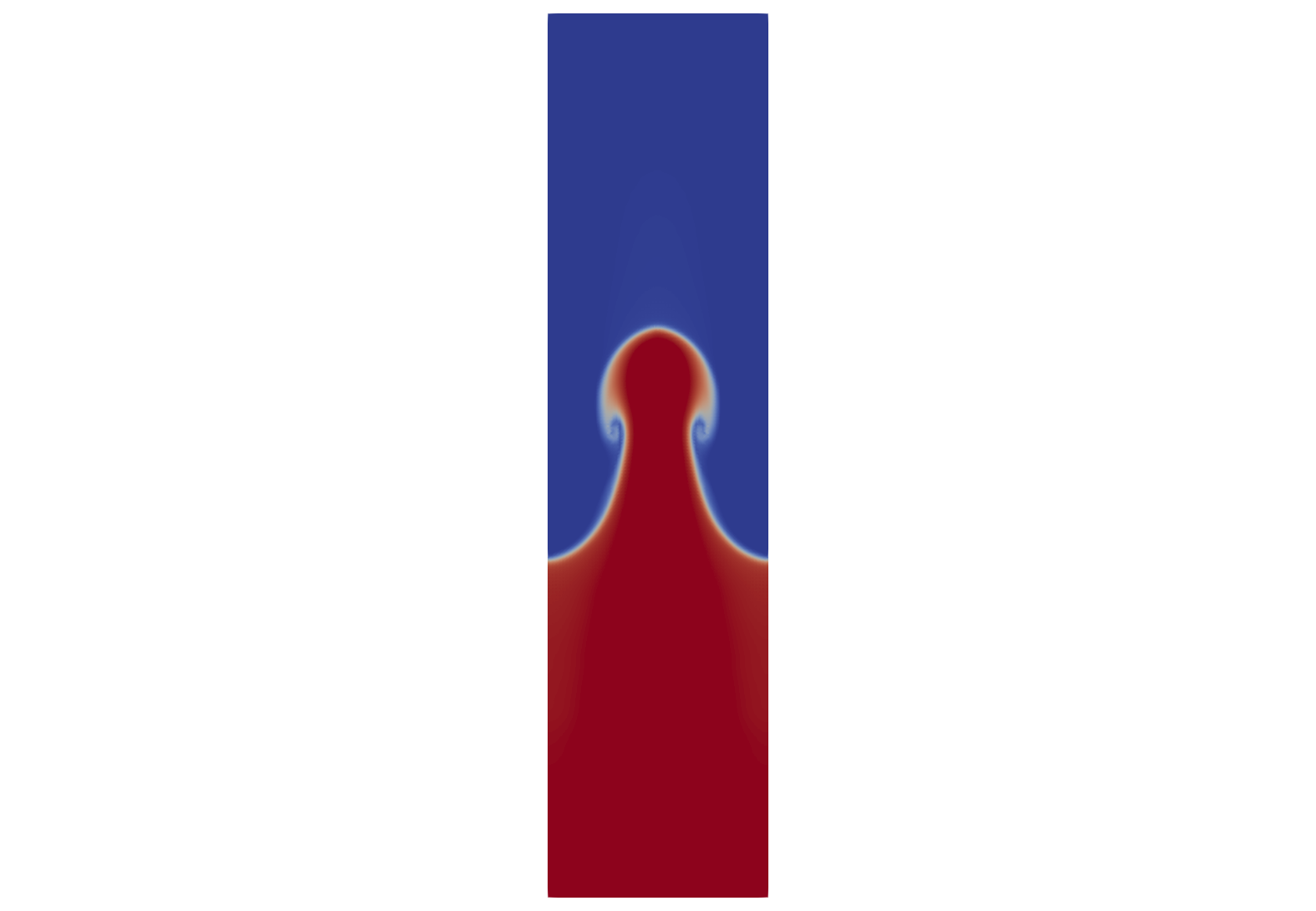}
\includegraphics[scale=0.25,trim=9in 2in 9in 2in,clip=true]{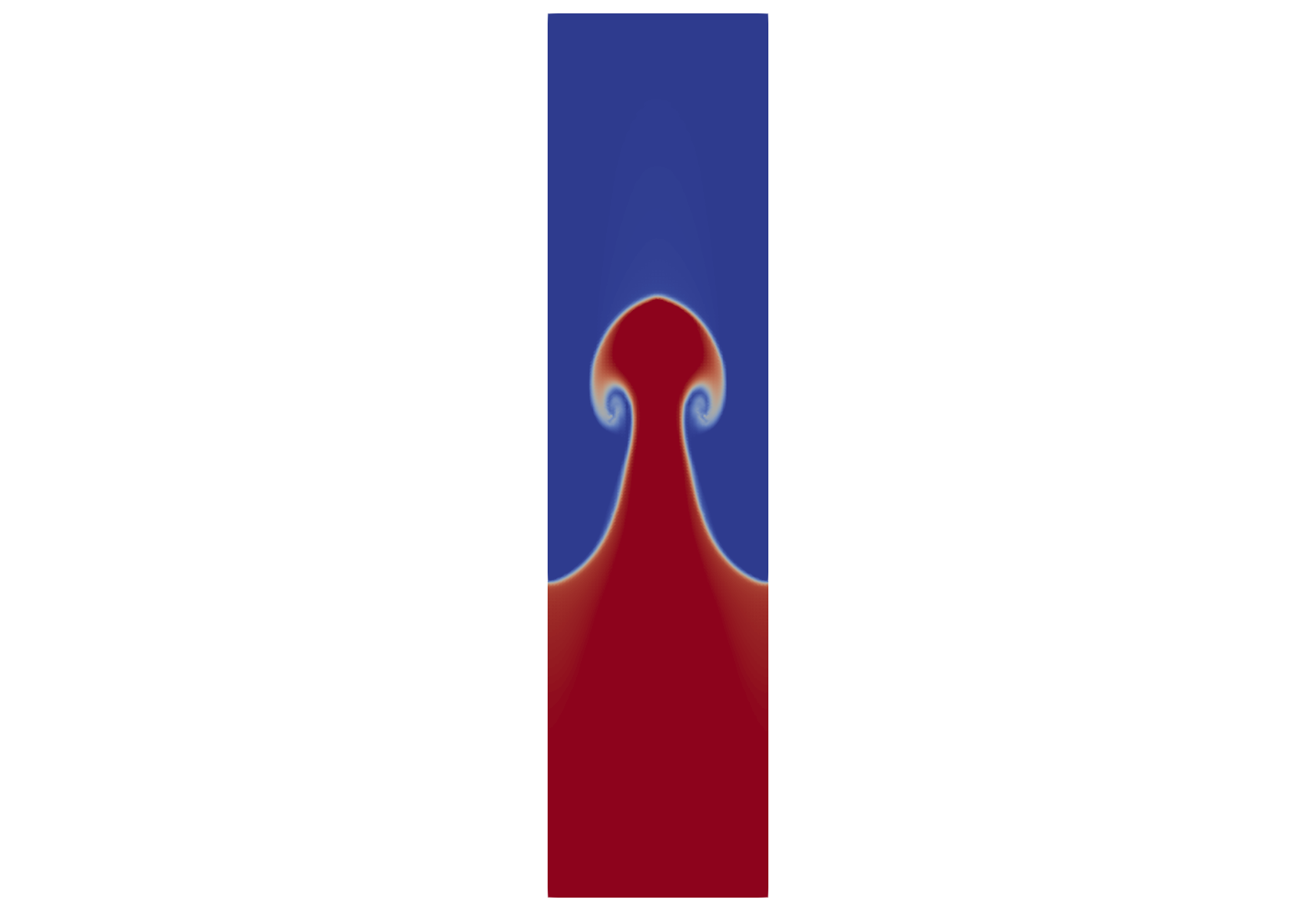}
\includegraphics[scale=0.25,trim=9in 2in 9in 2in,clip=true]{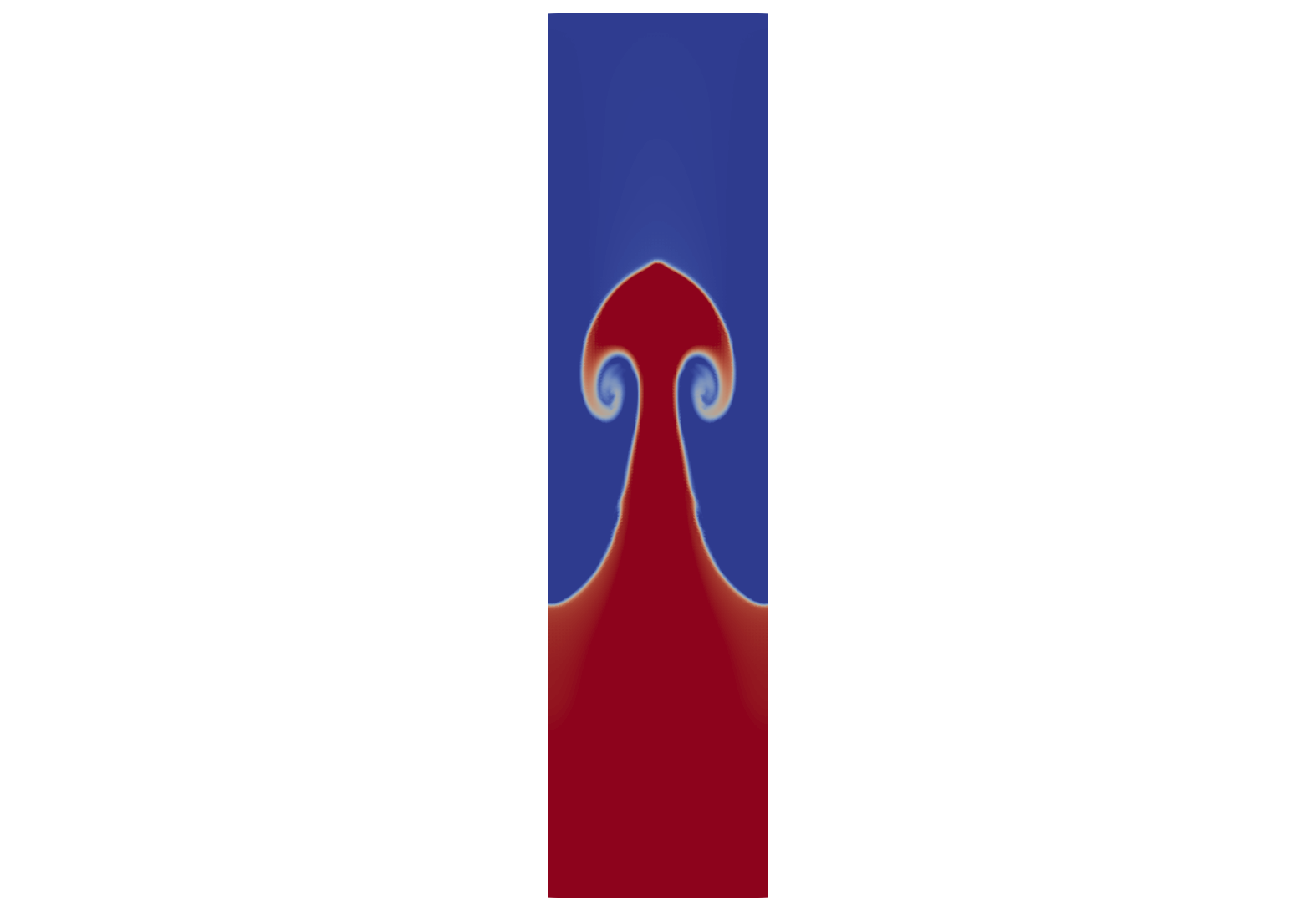}
\includegraphics[scale=0.25,trim=9in 2in 9in 2in,clip=true]{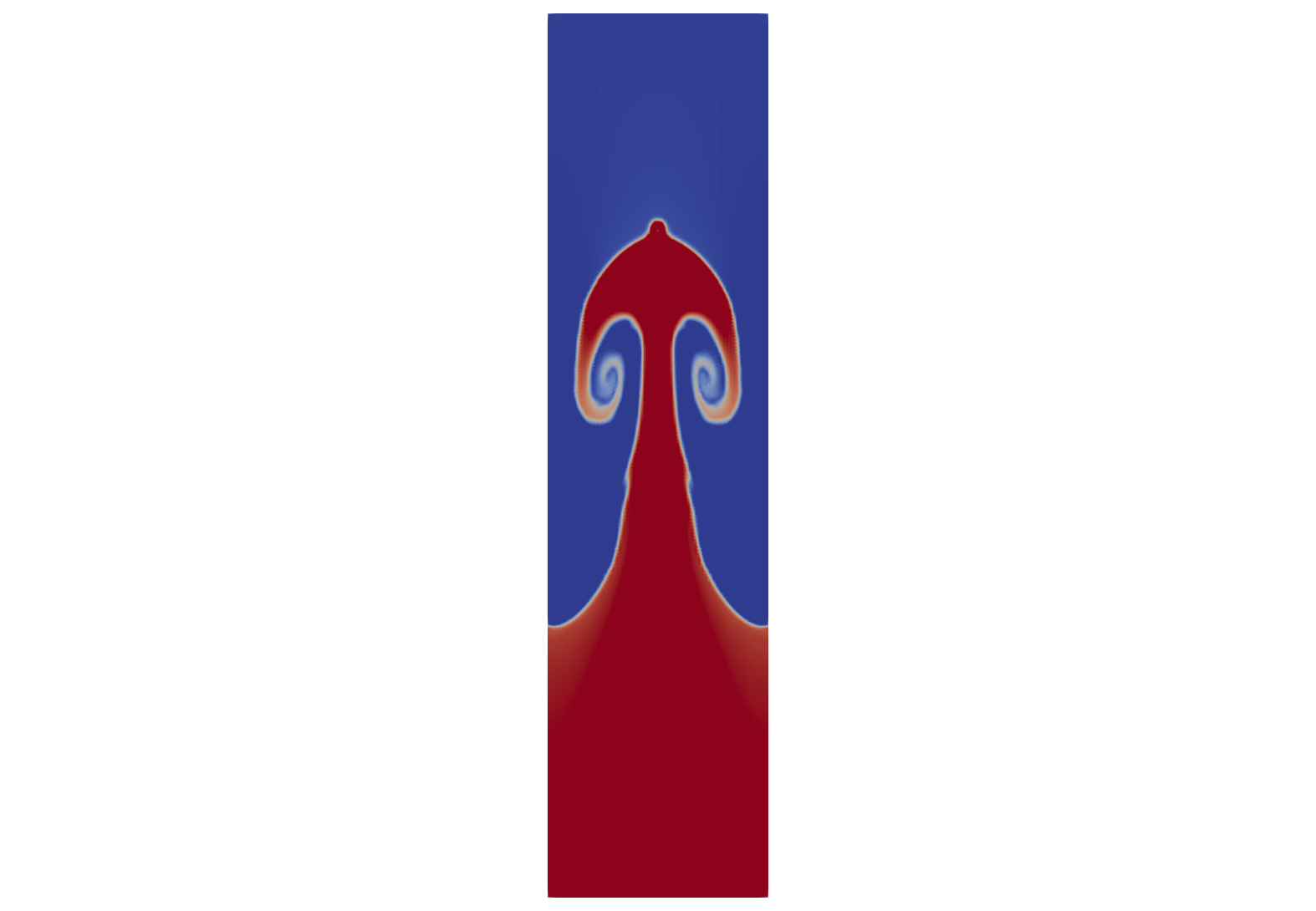}
\caption{Contours of the mass density at $t=1.0,1.2,1.4,1.6,1.8,2.0$ in the Rayleigh-Taylor instability simulation.}
\label{fig:raytay}
\end{figure}

\begin{figure}
\centering
\begin{tikzpicture}[scale=1]
\begin{axis}[xlabel=$t$,ylabel=$E(t)/E(0)-1$]
\addplot[blue] table [x expr=\coordindex*0.01, y expr=\thisrowno{0}]{raytay_energyminus1.dat};
\end{axis}
\end{tikzpicture}
\caption{Relative errors in the energy $E(t) = \int_\Omega \left( \frac{1}{2} \rho |u|^2 + \rho e(\rho) + \rho \Phi \right) \, {\rm d}x$ during the Rayleigh-Taylor instability simulation.}
\label{fig:raytay_energy}
\end{figure}
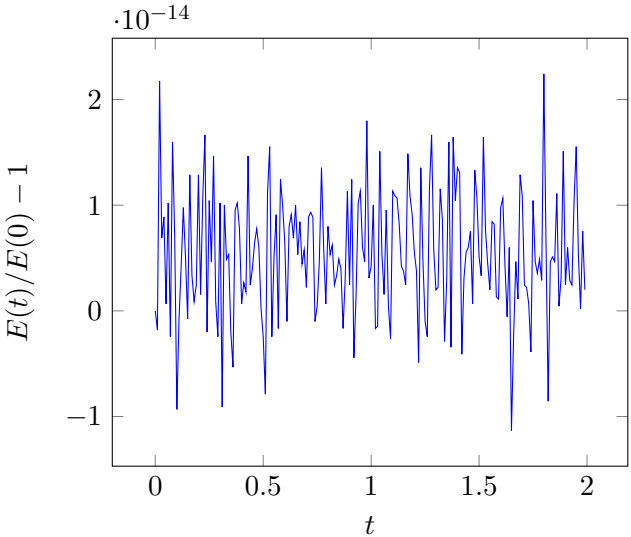

\paragraph{Convergence.}

To test our numerical method, we used~(\ref{FEMvelocity}-\ref{FEMdensity}) to simulate a rotating fluid with angular velocity $\omega = 1$ (i.e. $R=(-y,x)$) and internal energy $e(\rho) = \frac{1}{2} \rho$ in the absence of a gravitational field. This choice of the function $e(\rho)$ corresponds to the case of the rotating shallow water equations, for which $\rho$ is interpreted as the fluid depth. We initialized $u(x,y,0) = (0,0)$ and $\rho(x,y,0) = 2 + \sin(\pi x/2) \sin(\pi y/2)$ on $\Omega=(-1,1) \times (-1,1)$ and numerically integrated~(\ref{FEMvelocity}-\ref{FEMdensity}) using the energy-preserving time discretization~(\ref{energy_pres_A}-\ref{energy_pres_D}) with $\Delta t = 0.00625$.  We used the finite element spaces $R_h=RT_r(\mathcal{T}_h)$ and $V_h^r$ with $r=0,1,2$ on a uniform triangulation $\mathcal{T}_h$ of $\Omega$ with maximum element diameter $h = 2^{-j}$, $j=0,1,2,3$.  We computed the $L^2$-errors in the velocity and density at time $T=0.5$ by comparing with an ``exact solution'' obtained with $h=2^{-5}$, $r=2$.  The results in Table~\ref{tab:a0b0} indicate that the method's convergence order is optimal (order $r+1$) when $r=0$ and suboptimal when $r>0$, but still grows with $r$.

We also repeated the above experiment with varying values of $\Delta t$ and with fixed values of $h=2^{-4}$ and $r=2$.  The results in Table~\ref{tab:a0b0_dt} indicate that the method is second-order accurate with respect to $\Delta t$.

\paragraph{Rayleigh-Taylor instability.}

Next, we simulated a Rayleigh-Taylor instability.  For this test, we considered a fully (or baroclinic) compressible fluid, whose energy depends on both the mass density $\rho$ and the entropy density $s$, both of which are advected parameters.  The setting is the same as above, but with a Lagrangian
\begin{equation}\label{Baroclinic_L}
\ell(u,\rho,s)= \int_\Omega \Big[\frac{1}{2}\rho|u|^2  - \rho e(\rho, \eta) -\rho\Phi\Big] {\rm d}x,
\end{equation}
where $\eta= \frac{s}{\rho}$ is the specific entropy.  In terms of the discrete velocity $A \in \Delta_h^R$, discrete mass density $D \in V_h^r$, and discrete entropy density $S \in V_h^r$, the spatially discrete Euler-Poincar\'e equations for this Lagrangian read
\begin{equation}\label{EP_a_weak_fluid_S}
\Big\langle\!\!\Big\langle \partial_t \frac{\delta\ell_d}{\delta A}, B\Big\rangle\!\!\Big\rangle + \Big\langle\!\!\Big\langle \frac{\delta\ell_d}{\delta A}, [A,B]\Big\rangle\!\!\Big\rangle + \Big\langle \frac{\delta\ell_d}{\delta D},D\cdot B \Big\rangle+\Big\langle \frac{\delta\ell_d}{\delta S},S\cdot B \Big\rangle=0, \quad \forall \;t \in (0,T),\quad\forall\; B\in \Delta_h^R,
\end{equation}
or, equivalently,
\begin{equation}\label{EP_a_strong_fluid_S}
 \partial_t \frac{\delta\ell_d}{\delta A}+ \operatorname{ad}^*_A \frac{\delta\ell_d}{\delta A} - \frac{\delta\ell_d}{\delta D}\diamond D-  \frac{\delta\ell_d}{\delta S}\diamond S\in (\Delta_h^R)^\circ , \quad \forall t \in (0,T),
\end{equation}
together with
\begin{equation} \label{Devolution}
\langle \partial_t D, E \rangle + \langle D, AE \rangle = 0,\quad\text{and}\quad \langle \partial_t S, H \rangle + \langle S, AH \rangle = 0, \quad \forall t \in (0,T),\quad\forall\; E,H\in V_h^r.
\end{equation}
In analogy with~(\ref{energy_pres_A}-\ref{energy_pres_D}), an energy-preserving time discretization is given by
\begin{align}
&&\Big\langle\!\!\Big\langle \frac{1}{\Delta t}\left(\frac{\delta\ell_d}{\delta A_k}- \frac{\delta\ell_d}{\delta A_{k-1}}\right) , B_k \Big\rangle\!\!\Big\rangle +\frac{1}{2} \Big\langle\!\!\Big\langle \frac{\delta\ell_d}{\delta A_{k-1}} + \frac{\delta\ell_d}{\delta A_k}, [A_{k-1/2},B_k] \Big\rangle\!\!\Big\rangle \hspace{0in}& \nonumber \\
&&  + \langle F_{k-1/2}, D_{k-1/2} \cdot B_k\rangle +  \langle G_{k-1/2}, S_{k-1/2} \cdot B_k\rangle&=0,\quad \forall\; B_k \in \Delta_h^R, \label{baro_energy_pres_A} \\
&&\left\langle \frac{D_k - D_{k-1}}{\Delta t}, E_k \right\rangle + \langle D_{k-1/2} \cdot A_{k-1/2}, E_k \rangle &= 0, \quad \forall E_k \in V_h^r, \label{baro_energy_pres_D} \\
&&\left\langle \frac{S_k - S_{k-1}}{\Delta t}, H_k \right\rangle + \langle S_{k-1/2} \cdot A_{k-1/2}, H_k \rangle &= 0, \quad \forall H_k \in V_h^r, \label{baro_energy_pres_S}
\end{align}
where $A_{k-1/2} = \frac{1}{2}(A_{k-1}+A_k)$, $D_{k-1/2} = \frac{1}{2}(D_{k-1}+D_k)$, $S_{k-1/2} = \frac{1}{2}(S_{k-1}+S_k)$,
\begin{align}
F_{k-1/2} &= \frac{1}{2} \widehat{A_{k-1}} \cdot \widehat{A_k} - \frac{1}{2}\left(f(D_{k-1},D_k,S_{k-1})+f(D_{k-1},D_k,S_k) \right) - \Phi, \label{baro_discrete_gradient1} \\
G_{k-1/2} &= -\frac{1}{2}\left( g(S_{k-1},S_k,D_{k-1}) + g(S_{k-1},S_k,D_k)\right), \label{baro_discrete_gradient2}
\end{align}
and
\begin{align*}
f(D,D',S) &= \frac{D'e(D',S/D')-De(D,S/D)}{D'-D}, \\
g(S,S',D) &= \frac{De(D,S'/D)-De(D,S/D)}{S'-S}.
\end{align*}
We took $e$ equal to the internal energy for a perfect gas,
\[
e(\rho,\eta) =K e^{\eta/C_v} \rho^{\gamma-1},
\]
where $\gamma=5/3$ and $K=C_v=1$, and we used a graviational potential $\Phi = -y$, which corresponds to an upward gravitational force.  We initialized
\begin{align*}
\rho(x,y,0) &= 1.5-0.5\tanh\left( \frac{y-0.5}{0.02} \right), \\
u(x,y,0) &= \left( 0, -0.025\sqrt{\frac{\gamma p(x,y)}{\rho(x,y,0)}} \cos(8\pi x) \exp\left( -\frac{(y-0.5)^2}{0.09}\right) \right), \\
s(x,y,0) &= C_v \rho(x,y,0) \log\left( \frac{p(x,y)}{(\gamma-1)K\rho(x,y,0)^\gamma} \right),
\end{align*}
where 
\[
p(x,y) = 1.5y+1.25 + (0.25-0.5y) \tanh\left( \frac{y-0.5}{0.02} \right).
\]
We implemented~(\ref{baro_energy_pres_A}-\ref{baro_energy_pres_S}) with $\Delta t = 0.01$ and with the finite element spaces $R_h = RT_0(\mathcal{T}_h)$ and $V_h^1$ on a uniform triangulation $\mathcal{T}_h$ of $\Omega = (0,1/4) \times (0,1)$ with maximum element diameter $h = 2^{-8}$.  
We incorporated upwinding into~(\ref{baro_energy_pres_A}-\ref{baro_energy_pres_S}) using the strategy detailed in~\cite{GaGB2019}, which retains the scheme's energy-preserving property.  
Plots of the computed mass density at various times $t$ are shown in Fig.~\ref{fig:raytay}.  Fig.~\ref{fig:raytay_energy} confirms that energy was preserved exactly up to roundoff errors.

\appendix

\section{Euler-Poincar\'e variational principle}\label{Appendix_A}

In this Appendix we first recall the Euler-Poincar\'e principle for invariant Euler-Lagrange systems on Lie groups. This general setting underlies the Lie group description of incompressible flows recalled in \S\ref{subsec_incomp} due to \cite{Ar1966}, in which case the Lie group is $G=\operatorname{Diff}_{\rm vol}(\Omega)$. It also underlies the semidiscrete setting, in which case the Lie group is $G=G_h$. In this situation, however, a nonholonomic constraint needs to be considered, see Appendix \ref{Appendix_B}. Then, we describe the extension of this setting that is needed to formulate the variational formulation of compressible flow and its discretization.

\subsection{Euler-Poincar\'e variational principle for incompressible flows}\label{A1}

Let $G$ be a Lie group and let $L:TG \rightarrow \mathbb{R}$ be a Lagrangian defined on the tangent bundle $TG$ of $G$. The associated equations of evolution, given by the Euler-Lagrange equations, arise as the critical curve condition 
for the Hamilton principle
\begin{equation}\label{HP_G}
\delta\int_0^T L(g(t), \dot g(t)){\rm d} t=0,
\end{equation}
for arbitrary variations $\delta g$ with $\delta g(0)= \delta g(T)=0$.

\medskip

If we assume that $L$ is $G$-invariant, i.e., $L(gh, \dot gh)=L(g, \dot g)$, for all $h\in G$, then $L$ induces a function $\ell:\mathfrak{g}\rightarrow \mathbb{R}$ on the Lie algebra $\mathfrak{g}$ of $G$, defined by $\ell(u)=L(g,\dot g)$, with $u=\dot g g^{-1}\in \mathfrak{g}$. In this case the equations of motion can be expressed exclusively in terms of $u$ and $\ell$ and are obtained by rewriting the variational principle \eqref{HP_G} in terms of $\ell$ and $u(t)$. One gets
\begin{equation}\label{EP}
\delta\int_0^T \ell(u(t)){\rm d}t=0,\quad \text{for $\delta u = \partial_t v + [v,u]$},
\end{equation}
where $v(t)\in \mathfrak{g}$ is an arbitrary curve with $v(0)=v(T)=0$. The form of the variation $\delta u$ in \eqref{EP} is obtained by a direct computation using $u= \dot g g^{-1}$ and defining $v= \delta g g^{-1}$. 

\medskip

In order to formulate the equations associated to \eqref{EP} one needs to select an appropriate space in nondegenerate duality with $\mathfrak{g}$ denoted $\mathfrak{g}^*$ (the usual dual space in finite dimensions). We shall denote by $\langle\!\langle\,, \rangle \!\rangle: \mathfrak{g}^* \times \mathfrak{g}\rightarrow \mathbb{R}$ the associated nondegenerate duality pairing. From \eqref{EP} one directly obtains the equation
\begin{equation}\label{EP_weak}
\Big\langle\!\!\Big\langle \partial_t \frac{\delta\ell}{\delta u}, v\Big\rangle\!\!\Big\rangle + \Big\langle\!\!\Big\langle \frac{\delta\ell}{\delta u}, [u,v]\Big\rangle\!\!\Big\rangle =0,\quad\forall\; v\in \mathfrak{g}.
\end{equation}
In \eqref{EP_weak}, the functional derivative $\frac{\delta\ell}{\delta u}\in \mathfrak{g}^*$ of $\ell$ is defined in terms of the duality pairing as
\[
\Big\langle\!\!\Big \langle\frac{\delta\ell}{\delta u}, \delta u \Big\rangle\!\!\Big\rangle= \left.\frac{d}{d\epsilon}\right|_{\epsilon=0}\ell(u+\epsilon\delta u).
\]

In finite dimensions, and under appropriate choices for the functional spaces in infinite dimensions, \eqref{EP_weak} is equivalent to the Euler-Poincar\'e equation
\[
\partial_t \frac{\delta\ell}{\delta u}+ \operatorname{ad}^*_u\frac{\delta\ell}{\delta u} =0,
\]
where the coadjoint operator $ \operatorname{ad}^*_u: \mathfrak{g}^* \rightarrow \mathfrak{g}^*$ is defined by $\langle\!\langle \operatorname{ad}^*_um,v\rangle \!\rangle=\langle\!\langle m,[u,v]\rangle \!\rangle$.

\medskip

For incompressible flows, without describing the functional analytic setting for simplicity, we have $G=\operatorname{Diff}_{\rm vol}(\Omega)$ and $\mathfrak{g}= \mathfrak{X}_{\rm vol}(\Omega)$ the Lie algebra of divergence free vector fields parallel to the boundary. We can choose $\mathfrak{g}^*= \mathfrak{X}_{\rm vol}(\Omega)$ with duality pairing $\langle\!\langle\,, \rangle \!\rangle$ given by the $L^2$ inner product. A direct computation gives the coadjoint operator $\operatorname{ad}^*_u m= \mathbf{P}(u \cdot \nabla m + \nabla u ^\mathsf{T} m)$, where $\mathbf{P}$ is the Leray-Hodge projector onto $\mathfrak{X}_{\rm vol}(\Omega)$. One directly checks that in this case \eqref{EP} yields the Euler equations \eqref{Euler_equations} for incompressible flows.

\subsection{Euler-Poincar\'e variational principle for compressible flows}\label{A2}

The general setting underlying the variational formulation for compressible fluids starts exactly as before, namely, a system whose evolution is given by the Euler-Lagrange equations for a Lagrangian defined on the tangent bundle of a Lie group $G$. The main difference is that the Lagrangian depends parametrically on some element $a_0\in V$ of a vector space (the reference mass density $\varrho_0$ in the case of the barotropic compressible fluid, the reference mass and entropy densities $\varrho_0$ and $S_0$ for the general compressible fluid) on which $G$ acts by representation, and, in addition, $L$ is invariant only under the subgroup of $G$ that keeps $a_0$ fixed. If we denote by $L(g, \dot g, a_0)$ this Lagrangian and by $a\in V\mapsto a\cdot g \in V$ the representation of $G$ on $V$, the reduced Lagrangian is defined by $\ell(u, a)= L(g, \dot g, a_0)$, where $u= \dot gg^{-1}$, $a= a_0\cdot g^{-1}$.

\medskip

The Hamilton principle now yields the variational formulation
\begin{equation}\label{EP_a}
\delta\int_0^T \ell(u(t),a(t)){\rm d}t=0,\quad \text{for $\delta u = \partial_t v + [v,u]$ and $\delta a= - a\cdot v$},
\end{equation}
where $v(t)\in \mathfrak{g}$ is an arbitrary curve with $v(0)=v(T)=0$. The form of the variation $\delta u$ in \eqref{EP_a} is the same as before, while the expression for $\delta  a$ is obtained from the relation $a= a_0\cdot g^{-1}$.

\medskip

From \eqref{EP_a} and with respect to the choice of a spaces $\mathfrak{g}^*$ and $V^*$ in nondegenerate duality with $\mathfrak{g}$ and $V$, with duality pairings $\langle\!\langle \,,\rangle\!\rangle$ and $\langle\,,\rangle_V$, one directly obtains the equations
\begin{equation}\label{EP_a_weak}
\Big\langle\!\!\Big\langle \partial_t \frac{\delta\ell}{\delta u}, v\Big\rangle\!\!\Big\rangle + \Big\langle\!\!\Big\langle \frac{\delta\ell}{\delta u}, [u,v]\Big\rangle\!\!\Big\rangle + \Big\langle \frac{\delta\ell}{\delta a},a\cdot v \Big\rangle_V=0,\quad\forall\; v\in \mathfrak{g}.
\end{equation}
The continuity equation
\[
\partial_t a + a \cdot u=0
\]
arises from the definition $a(t)= a_0 \cdot g(t)^{-1}$. In a similar way with above, \eqref{EP_a_weak} now yields the Euler-Poincar\'e equations
\begin{equation}\label{EP_a_strong}
\partial_t \frac{\delta\ell}{\delta u}+ \operatorname{ad}^*_u\frac{\delta\ell}{\delta u} =\frac{\delta \ell}{\delta a}\diamond a,
\end{equation}
where $\frac{\delta \ell}{\delta a}\diamond a\in \mathfrak{g}^*$ is defined by $\left\langle\!\left\langle\frac{\delta \ell}{\delta a}\diamond a,v\right\rangle\!\right\rangle = -\Big\langle \frac{\delta\ell}{\delta a},a\cdot v \Big\rangle_V $, for all $v\in \mathfrak{g}$. We refer to \cite{HoMaRa1998} for a detailed exposition.

\medskip

For the compressible fluid, in the continuous case we have $G= \operatorname{Diff}(\Omega)$ and $\mathfrak{g}=\mathfrak{X}(\Omega)$ the Lie algebra of vector fields on $\Omega$ with vanishing normal component to the boundary. We choose to identify $\mathfrak{g}^*$ with $\mathfrak{g}$ via the $L^2$ duality pairing. Consider the Lagrangian \eqref{Baroclinic_L} of the general compressible fluid. Using the expressions $\operatorname{ad}^*_u m = u \cdot \nabla m + \nabla u^\mathsf{T} m + m \operatorname{div} u$, $\frac{\delta \ell}{\delta u}= \rho u$, $\frac{\delta \ell}{\delta \rho}= \frac{1}{2}|u|^2 - e(\rho) - \rho \frac{\partial e}{\partial \rho}+\eta \frac{\partial e}{\partial\eta}-\Phi$, $\frac{\delta \ell}{\delta s}= - \frac{\partial e}{\partial\eta}$, $\frac{\delta\ell}{\delta \rho}\diamond \rho= \rho \nabla \frac{\delta\ell}{\delta \rho}$, and $\frac{\delta\ell}{\delta s}\diamond s= s \nabla \frac{\delta\ell}{\delta s}$, one directly obtains
\[
\rho(\partial_tu + u \cdot\nabla u)= - \nabla p - \rho\nabla\Phi
\]
from \eqref{EP_a_strong}, with $p= \rho^2\frac{\partial e}{\partial \rho}$. 
For the semidiscrete case, one uses a nonholonomic version of the Euler-Poincar\'e equations \eqref{EP_a_strong}, reviewed in the next paragraph.

\section{Remarks on the nonholonomic Euler-Poincar\'e variational formulation}\label{Appendix_B}

Hamilton's principle can be extended to the case in which the system under consideration is subject to a constraint, given by a distribution on the configuration manifold, i.e., a vector subbundle of the tangent bundle. This is known as the Lagrange-d'Alembert principle and, for a system on a Lie group $G$ and constraint $\Delta_G \subset TG$, it is given by the same critical condition \eqref{HP_G} but only with respect to variations statisfaying the constraint, i.e.,  $\delta g\in \Delta_G$.

\medskip

In the $G$-invariant setting recalled in \S\ref{A1} it is assumed that the constraint $\Delta_G$ is also $G$-invariant and thus induces a subspace $\Delta \subset \mathfrak{g}$ of the Lie algebra. In the more general setting of \S\ref{A2}, one can allow $\Delta_G$ to be only $G_{a_0}$-invariant, although for the situation of interest in this paper, $\Delta_G$ is also $G$-invariant. 

\medskip

The Lagrange-d'Alembert principle yields now the Euler-Poincar\'e-d'Alembert principle \eqref{EP_a} in which we have the additional constraint $u(t)\in \Delta$ on the solution and $v(t)\in \Delta$ on the variations, so that \eqref{EP_a_weak} becomes
\begin{equation}\label{EP_a_weak_NH}
\Big\langle\!\!\Big\langle \partial_t \frac{\delta\ell}{\delta u}, v\Big\rangle\!\!\Big\rangle + \Big\langle\!\!\Big\langle \frac{\delta\ell}{\delta u}, [u,v]\Big\rangle\!\!\Big\rangle + \Big\langle \frac{\delta\ell}{\delta a},a\cdot v \Big\rangle_V=0,\quad\text{for all $v\in \Delta$, where $u\in \Delta$}.
\end{equation}
In presence of the nonholonomic constraint, \eqref{EP_a_strong} becomes
\begin{equation}\label{EP_a_strong_NH}
\partial_t \frac{\delta\ell}{\delta u}+ \operatorname{ad}^*_u\frac{\delta\ell}{\delta u} -\frac{\delta \ell}{\delta a}\diamond a\in \Delta^\circ,\qquad u \in \Delta,
\end{equation}
where $\Delta^\circ= \{m \in \mathfrak{g}^* \mid \langle\!\langle m,u\rangle\!\rangle=0, \;\forall\; u\in \Delta\}$.

\medskip

There are two important remarks concerning \eqref{EP_a_weak_NH} and \eqref{EP_a_strong} that play an important role for the variational discretization carried out in this paper.
First, we note that although the solution belongs to the constraint, i.e., $u\in \Delta$, the equations depend on the expression of the Lagrangian $\ell$ on a larger space, namely, on $\Delta+ [\Delta, \Delta]$. It is not enough to have its expression only on $\Delta$. This is a main characteristic of nonholonomic mechanics. 
Second, a sufficient condition to get a solvable differential equation is that the map $u \in \Delta \mapsto \frac{\delta \ell}{\delta u } \in \mathfrak{g}^*/\Delta^\circ$ is a diffeomorphism for all $a$.

\medskip

\section{Polynomials}\label{Appendix_C}

Below we prove two facts about polynomials that are used in the proof of Proposition~\ref{important_prop}. We denote by $H_r(K)$ the space of homogeneous polynomials of degree $r$ on a simplex $K$.  To distinguish powers from indices, we denote coordinates by $x_1,x_2,\dots,x_n$ rather than $x^1,x^2,\dots,x^n$ in this section.

\begin{lemma} \label{lemma:PrPr}
Let $K$ be a simplex of dimension $n \ge 1$.  For every integer $r \ge 0$,
\[
\left\{ \sum_{i=1}^N p_i q_i \mid N \in \mathbb{N}, \,  p_i, q_i \in P_r(K), \, i=1,2,\dots,N \right\} = P_{2r}(K).
\]
\end{lemma}
\begin{proof}
This follows from the fact that every monomial in $P_{2r}(K)$ can be written as a product of two monomials in $P_r(K)$.
\end{proof}

\begin{lemma} \label{lemma:PrgradPr}
Let $K$ be a simplex of dimension $n \in \{2,3\}$.  For every integer $r \ge 0$,
\[
\left\{ \sum_{i=1}^N p_i \nabla q_i \mid N \in \mathbb{N}, \,  p_i, q_i \in P_r(K), \, i=1,2,\dots,N \right\} = P_{2r-1}(K)^n.
\]
\end{lemma}
\begin{proof}
We proceed by induction.

Denote $Q_r(K) = \left\{ \sum_{i=1}^N p_i \nabla q_i \mid N \in \mathbb{N}, \,  p_i, q_i \in P_r(K), \, i=1,2,\dots,N \right\}$.  By inductive hypothesis, $Q_r(K)$ contains $Q_{r-1}(K)=P_{2r-3}(K)^n$.  It also contains $H_{2r-2}(K)^n$.  Indeed, if $f e_k \in H_{2r-2}(K)^n$ with $k \in \{1,2,\dots,n\}$ and $f$ a monomial, then $f=x_j g$ for some $g \in H_{2r-3}(K)$ and some $j \in \{1,2,\dots,n\}$, so $fe_k = \sum_i (x_j p_i) \nabla q_i \in Q_r(K)$ for some $p_i,q_i \in P_{r-1}(K)$ by inductive hypothesis.  Thus, $Q_r(K)$ contains $P_{2r-2}(K)^n$.  

Next we show that $Q_r(K)$ contains every $u \in H_{2r-1}(K)^n$.  Without loss of generality, we may assume $u = fe_1$ with $f \in H_{2r-1}(K)$ a monomial.  When $n=2$, the only such vector fields are $u = x_1^a x_2^{2r-1-a} e_1$, $a=0,1,\dots,2r-1$, which can be expressed as
\[
x_1^a x_2^{2r-1-a} e_1 = \begin{cases}
\frac{1}{r} x_1^{a-r+1} x_2^{2r-1-a} \nabla (x_1^r), &\mbox{ if } a \ge r-1, \\
\frac{1}{a+1} x_2^r \nabla ( x_1^{a+1} x_2^{r-1-a} ) - \frac{r-1-a}{(a+1)r} x_1^{a+1} x_2^{r-1-a} \nabla (x_2^r),  &\mbox{ if } a < r-1.
\end{cases}
\]
The case $n=3$ is handled similarly by considering the vector fields $f e_1$ with 
\[
f \in \{x_1^a x_2^b x_3^{2r-1-a-b} \mid a,b \ge 0, \, a+b \le 2r-1 \}.
\]
\end{proof}

\end{document}